\documentclass[11pt, reqno, a4paper, final]{amsart}

\usepackage[applemac]{inputenc}      % denne pakke giver mulighed for ? ?og ?med TeXShop
\usepackage[english]{babel}
\usepackage[T1]{fontenc}
\usepackage{amsmath,amsfonts,amssymb,amsthm,amscd}
\usepackage{eucal}
\usepackage{array}
\usepackage{mathtools}
\usepackage{color}
\usepackage{bbm}
\usepackage{graphicx}
\usepackage{hyperref}
\usepackage{a4wide}										 	% Smal margen
\usepackage{mathrsfs}									  % Ekstra kroellede bogstaver
\usepackage[all]{xy}										% Laekre kommutative diagrammer
\usepackage{url}												% %Inds?ttelse af www-links
\usepackage{verbatim,epsfig}
\usepackage[notref,notcite,draft]{showkeys}   %viser hvad man har kaldt sine labels ude i margen
\usepackage{xfrac} % makes it possible to do slanted fractions \sfrac{}{}
\usepackage[linewidth=1pt]{mdframed}

\newcounter{claim-counter}
\theoremstyle{plain}

\newtheorem{thm}{Theorem}[section] 
\newtheorem{theorem}{Theorem}[section]

\newtheorem{lem}[thm]{Lemma}
\newtheorem{lemma}[thm]{Lemma}
\newtheorem{prop-defi}[thm]{Definition \& Proposition}
\newtheorem{prop}[thm]{Proposition}

\newtheorem*{thm*}{Theorem}
\newtheorem*{prop*}{Proposition}
\newtheorem*{cor*}{Corollary}

\usepackage{thmtools}
\declaretheorem[style=theorem,name={Theorem}]{theoremletter}

\newtheorem{introcorollary}[theoremletter]{Corollary}

\theoremstyle{definition}
\newtheorem{defi}[thm]{Definition}

\newtheorem{rem}[thm]{Remark}
\newtheorem{remark}[thm]{Remark}

\newtheorem*{claim*}{Claim}

\newcommand{\NN}{{\mathbb N}}

\newcommand{\RR}{{\mathbb R}}
\newcommand{\CC}{{\mathbb C}}

\newcommand{\MM}{{\mathbb M}}

\newcommand{\F}{{\mathcal F}}
\newcommand{\E}{{\mathcal E}}

\renewcommand{\H}{\mathcal{H}}
\newcommand{\I}{\mathcal{I}}
\renewcommand{\L}{{\mathcal L}}

\newcommand{\tens}{\otimes}

\newcommand{\To}{\longrightarrow}

\newcommand{\del}{{\partial}}

\renewcommand{\d}{\operatorname{d}\hspace{-0.03cm}}

\renewcommand{\leq}{\leqslant}
\renewcommand{\geq}{\geqslant}

\newcommand{\im}{{\operatorname{im}}}

\newcommand{\loc}{{\operatorname{loc}}}

\renewcommand{\i}{{\operatorname{i}}}

\newcommand{\Cohom}{\operatorname{H}}

\renewcommand{\restriction}{\rvert}

\newcommand{\cl}{{\operatorname{cl}}}
\renewcommand{\restriction}{\mathord{\hspace{-0.05cm}\upharpoonright} }

\allowdisplaybreaks

\title{Cohomological induction and uniform measure equivalence}

\author{Thomas Gotfredsen}
\address{Thomas Gotfredsen, Department of Mathematics and Computer Science, University of Southern Denmark, Campusvej 55, DK-5230 Odense M, Denmark}
\email{thgot@imada.sdu.dk}

\author{David Kyed}
\address{David Kyed, Department of Mathematics and Computer Science, University of Southern Denmark, Campusvej 55, DK-5230 Odense M, Denmark}
\email{dkyed@imada.sdu.dk}


\begin{document}

\begin{abstract}
We construct a general cohomological induction isomorphism  from a uniform measure equivalence of  locally compact, second countable, unimodular groups which, as a special case, yields that the graded cohomology algebras of quasi-isometric, connected, simply connected, nilpotent Lie groups are isomorphic. This unifies results of Shalom and Sauer and also provides new insight into the quasi-isometry classification problem for low dimensional nilpotent Lie groups.
\end{abstract}

\maketitle

\section{Introduction}
Geometric and measurable group theory originated in the the works of Gromov and is by now a well developed and important tool in the study of countable discrete groups; see e.g.~\cite{gromov-infinite-grps, furman-gromovs-measure-equivalence, Furman-OE-rigidity, ornstein-weiss} and references therein. One of the many insights from the beginning of the present century, mainly due to Shalom,  Shalom-Monod \cite{monod-shalom} and Sauer \cite{sauer},  is that group cohomology, interacts well with one of the fundamental concepts in geometric group theory, namely that of \emph{quasi-isometry} \cite{roe-lecture-notes, yves-book}. These results are mainly concerned with the class of discrete, countable groups, but in recent years, increasing emphasis has been put on the the more general setting of  locally compact groups 
 \cite{bowen-hoff-ioana, BFS-integrable, carderi-le-maitre, bader-rosendal,  KKR17, KKR18}, 
 and results concerning the interplay between group cohomology  and  quasi-isometry are now beginning to emerge in this setting as well  \cite{Bourdon-Remy, Sauer-Schrodl}.
 The present article provides a contribution to this  line of research by proving that the central results due Shalom, Shalom-Monod and Sauer mentioned above,  admit natural generalisations to the class of unimodular, locally compact, second countable groups. Actually, our primary focus will not be on quasi-isometry but rather on the related notion known as (uniform) measure equivalence, which is a measurable analogue of quasi-isometry, originally
introduced for discrete groups by Gromov in  \cite{gromov-asymptotic-invariants}. Uniform measure equivalence can be defined also for unimodular, locally compact second countable groups \cite{BFS-integrable, KKR17}, and for  compactly generated, unimodular  groups it was shown in \cite{KKR17} that, just as in the discrete case,  uniform measure equivalence implies quasi-isometry and that the two notions coincide when the groups in question are amenable.   Our primary focus will therefore be on uniformly measure equivalent topological groups and our first main result, drawing inspiration from \cite[Proposition 4.6]{monod-shalom}, is the following reciprocity principle in group cohomology; for the basic definitions concerning cohomology and uniform measure equivalence, see Section \ref{sec:preliminaries}.

\begin{theoremletter}\label{mainthm:UME-induction}
If $G$ and $H$ are uniformly measure equivalent, locally compact, second countable, unimodular groups, then for any uniform measure equivalence coupling $(\Omega, \eta, X, \mu, Y,\nu, i,j)$ and any Fr{\'e}chet $G\times H$-module $\E$ there exists an isomorphism of topological vector spaces
\[
\Cohom^n\left(G, L^2_{\loc}(\Omega, \E)^H\right)\overset{\sim}{\To} \Cohom^n\left(H, L^2_{\loc}(\Omega,\E)^G\right),
\]
for all $n\geq 0$. In particular, when $\E=\RR$ with trivial $G\times H$-action, this induces an isomorphism
\[
 \Cohom^n(G, L^2(X)) \simeq \Cohom^n(H, L^2(Y)).
\]
\end{theoremletter}

As a consequence of Theorem \ref{mainthm:UME-induction} we also obtain a generalisation of \cite[Theorem 5.1]{sauer}. 
For the statement, recall that a group $G$ is said to have (Shalom's) property $H_T$ if $\underline{\Cohom}^1(G,\H)=\{0\}$ for every unitary $G$-module $\H$ with no non-trivial fixed-points; here $\underline{\Cohom}^1(G,\H)$ denotes the first \emph{reduced} cohomology of $G$ with coefficients in $\H$; i.e.~the maximal Hausdorff quotient of the ordinary cohomology $\Cohom^1(G,\H)$ (see Section \ref{sec:preliminaries} for more details).
In what follows, we will by \emph{property $H_T^n$} mean the obvious extension of property $H_T$ to higher cohomologies. Recall also that the total cohomology $\Cohom^*(G,\RR):= \oplus_{n\geq 0} \Cohom^n(G,\RR)$ and its reduced counterpart $\underline{\Cohom}^*(G,\RR):= \oplus_{n\geq 0} \underline{\Cohom}^n(G,\RR)$ become graded, unital $\RR$-algebras with respect to the cup product; see Sections \ref{subsec:cohomological-properties} and \ref{sec:cup} for further details on this. Our generalisation of \cite[Theorem 5.1]{sauer} can now be stated as follows:

\begin{theoremletter}\label{mainthm:ring-iso}
If $G$ and $H$ are uniformly measure equivalent, locally compact, second countable, unimodular groups {satisfying property $H_T^n$ for all $n\in \NN$} then the associated reduced cohomology algebras $\underline{\Cohom}^*(G,\RR)$ and $\underline{\Cohom}^*(H,\RR)$ are isomorphic as graded, unital $\RR$-algebras.
\end{theoremletter}

As already mentioned, for compactly generated amenable groups, quasi-isometry and uniform measure equivalence coincide, and
when the $n$-th Betti number $\beta^n(G):=\dim_{\RR} \Cohom^n(G,\RR)$ is finite, then  $\Cohom^n(G,\RR)$ is automatically Hausdorff  \cite[III, Prop.~3.1]{guichardet-book}, so that we indeed recover Sauer's result  \cite[Theorem 5.1]{sauer}.  
We remark that the class of connected, simply connected (csc) nilpotent Lie groups  (as well as the  class of finitely generated, nilpotent groups) satisfies the assumptions in Theorem \ref{mainthm:ring-iso} and that these furthermore have finite Betti numbers in all degrees; see Section \ref{subsec:cohomological-properties} for details on this. We therefore, in particular, also obtain the following.

\begin{introcorollary}\label{mainthm:cor}
If $G$ and $H$ are quasi-isometric, connected, simply connected,  nilpotent Lie groups then the cohomology algebras $\Cohom^*(G,\RR)$ and $\Cohom^*(H,\RR)$ are isomorphic as graded unital $\RR$-algebras.
\end{introcorollary}

For csc nilpotent Lie groups containing lattices, i.e.~those admitting a rational structure \cite{malcev, malcev-russion}, the isomorphisms in Theorem \ref{mainthm:ring-iso}  and Corollary \ref{mainthm:cor} were already known to experts in the field, as they can be be deduced from  \cite[Theorem 5.1]{sauer} via  \cite[Theorem 1]{katsumi}. In this way Theorem \ref{mainthm:ring-iso} and Corollary \ref{mainthm:cor} provide a more natural approach, covering also  csc nilpotent groups without lattices, and, as a special case, it also gives an isomorphism of graded $\RR$-algebras $\underline{\Cohom}^*(\Gamma,\RR)\simeq \underline{\Cohom}^*(G,\RR)$ for any locally compact second countable, unimodular group $G$ with a uniform lattice $\Gamma\leq G$; see \cite[Proposition 6.11]{KKR18}.\\

As an application of our results, we show in Section \ref{sec:application} how the results above can be used to  improve on the quasi-isometry classification programme for csc nilpotent Lie groups of dimension 7.
In particular, we show that for the seven 1-parameter families of 7-dimensional nilpotent Lie algebras (cf.~\cite{gong}), for which the corresponding csc nilpotent Lie groups are not distinguished by Pansu's theorem \cite[Th{\'e}or{\`e}me 3]{pansu},  it holds for five of the families that all (but at most finitely many) members have pairwise isomorphic cohomology algebras, but that the cohomology algebras of the two remaining families fall in uncountably many isomorphism classes, thus providing new examples of uncountable families of connected simply connected nilpotent Lie groups which are pairwise  non-quasi-isometric.

\subsubsection*{Standing assumptions}
Unless otherwise specified, all generic vector spaces will be over the reals. So if $(X,\mu)$ is a measure space, then $L^2(X)$ will denote the Hilbert space of real valued square integrable functions, and similarly for other function spaces.  We remark that this convention is primarily chosen to streamline notation, and that Theorem  \ref{mainthm:UME-induction}, Theorem \ref{mainthm:ring-iso} and Corollary \ref{mainthm:cor} hold verbatim over the complex numbers as well; see Remarks \ref{rem:cmplx-rem-1} and \ref{rem:cmplx-rem-2}.

\subsubsection*{Acknowledgements}
The authors gratefully acknowledge the financial support from  the Villum Foundation (grant no.~7423) and from the Independent Research Fund Denmark (grant no.~7014-00145B and 9040-00107B). They also thank Henrik D.~Petersen whose earlier joint work with D.K.~formed the basis for part of the present paper, Yves de Cornulier for suggesting the applications regarding  quasi-isometry classification of 7-dimensional csc nilpotent Lie groups, and Nicky Cordua Mattsson for his assistance with programming issues in Maple. Finally, the authors would also like to thank the two anonymous referees for many helpful comments and suggestions; in particular we are grateful for the valuable input regarding the applications to nilpotent Lie groups, which has  improved the  quality of  Section \ref{sec:application} significantly.

\section{Preliminaries}\label{sec:preliminaries}
\subsection{Group cohomology}
In this section we recall the basics on   cohomology theory for locally compact groups, from the point of view of relative homological algebra; the reader is referred to \cite{guichardet-book} for more details and proofs of the statements below. In what follows, $G$ denotes a locally compact second countable group.

\begin{defi}
A (continuous) $G$-module is a Hausdorff topological vector space $\E$ endowed with an action of $G$ by linear maps such that the action map $G\times \E\to \E$ is continuous.  
If $\E$ and $\F$ are $G$-modules, then a linear, continuous, $G$-equivariant map $\varphi\colon \E \to \F$ is called a  \emph{morphism of $G$-modules} and $\varphi$ is said to be strengthened if there exists a linear, continuous map $\eta\colon \F \to \E $ such that $\varphi\circ \eta \circ \varphi =\varphi$.   
\end{defi}
Note  that the map $\eta$ in the definition of a strengthened morphism is not required to be $G$-equivariant. The definition of a strengthened morphism given above is easily seen to be equivalent to the more standard formulation given, for instance, in \cite[Definition D.1]{guichardet-book}.

\begin{defi}
A $G$-module $\mathcal{I}$ is said to be \emph{relatively injective} if for any strengthened, one-to-one homomorphism $\iota\colon \E\to \F$ of $G$-modules $\E$ and $\F$ and any morphism $\varphi\colon \E \to \I$ there exists a morphism $\tilde{\varphi}\colon \F\to \I $ making the following diagram commutative \

\[
\xymatrix{
0\ar[r]  & \E  \ar[r]^{\iota}\ar[d]_{\varphi}  &  \F \ar[dl]^{\tilde{\varphi}} \\
           & \I &  
           }
\]

\end{defi} 

\begin{defi}
Let $\E$ be a $G$-module. A \emph{strengthened, relatively injective resolution} of $\E$  is an exact complex
\[
0\To \E \overset{\iota}{\To} {\I}^0 \overset{d^{0}}{\To} {\I}^1\overset{d^{1}}{\To} {\I}^2 \overset{d^{2}}{\To} \cdots
\]
such that each $\I^n$ is a relatively injective $G$-module and $\iota$ and $d^n$ are strengthened morphisms. 
\end{defi}

If $(\I^n,d^n)$ is  a strengthened,  relatively injective  resolution of a $G$-module $\E$, 
then the \emph{cohomology}, $\Cohom^n(G,E)$, and \emph{reduced cohomology}, $\underline{\Cohom}^n(G,E)$, of $G$ with coefficients in $\E$ is defined as that of the complex 
\[
0\To (\I^0)^{G} \overset{d^{0}\restriction}{\To} ({\I}^1)^G\overset{d^{1}\restriction}{\To} ({\I}^2)^G \overset{d^{2}\restriction}{\To} \cdots
\]
obtained by taking $G$-invariants; i.e.~
\[
\Cohom^n(G,\E)=\frac{\ker(d^n\restriction)}{\im(d^{n-1}\restriction)} \ \text{ and } \ \underline{\Cohom}^n(G,E)=\frac{\ker(d^n\restriction)}{\cl(\im(d^{n-1}\restriction))},
\]
where $\cl( - )$ denotes closure in the topological vector space $\I^n$. We will denote by $\Cohom^*(G,\E)$ and $\underline{\Cohom}^*(G,\E)$ the direct sums $\oplus_{n\geq 0} \Cohom^n(G,\E)$ and $\oplus_{n\geq 0} \underline{\Cohom}^n(G,\E)$, respectively.
 A standard argument \cite[Corollaire 3.1.1]{guichardet-book} shows that $\Cohom^n(G,\E)$ is independent of the choice of strengthened, relative injective resolution up to isomorphism of (not necessarily Hausdorff!) topological vector spaces and hence also $ \underline{\Cohom}^n(G,E)\simeq  {\Cohom}^n(G,E)/\cl(\{0\})$ is independent of the choice of resolution. For the cohomology to be well defined one of course needs to know that strengthened, relatively injective resolutions exist, but in \cite[III, Proposition 1.2]{guichardet-book} it is shown that the complex
\[
0\To \E \overset{\iota}{\To} C(G,\E)\overset{d^{0}}{\To} C(G^2,\E)\overset{d^{1}}{\To} C(G^3,\E) \overset{d^{2}}{\To} \cdots
\]
with  
\begin{align}\label{eq:standard-differential}
d^n(f)(g_0,\dots, g_{n+1})=\sum_{i=0}^{n+1}(-1)^i f(g_0,\dots, \hat{g}_{i},\dots, g_{n+1})
\end{align}
constitutes such a resolution. Here the circumflex denotes omission and $C(G^{n},\E)$ denotes the space of continuous functions from $G^n$ to $\E$ endowed with $G$-action given by
\begin{align}\label{eq:standard-action}
(g.f)(g_0,\dots, g_n ):=g(f(g^{-1}g_0,\dots, g^{-1}g_n)).
\end{align}
When $\E$ is a Fr{\'e}chet space and $X$ is a $\sigma$-compact, locally compact, topological space whose Borel $\sigma$-algebra is endowed with a measure $\mu$, then one defines, for  $p\in \NN$, $\L^p_{\loc}(X,\E)$ as those measurable functions $f\colon X\to \E$ such that
\[
\int_{K} q(f(x))^p \d \mu(x)<\infty,
\]
for every compact set $K\subset X$ and every continuous seminorm $q$ on $\E$, and $L^p_{\loc}(X,\E)$ is then defined by identifying functions in $\L^p_{\loc}(X,\E)$ that are equal $\mu$-almost everywhere. This too is a Fr{\'e}chet space when endowed with the topology defined by the family of seminorms
\[
Q_{K,q}(f):=\sqrt[\leftroot{-2}\uproot{2}p]{\int_{K} q(f(x))^p \d \mu(x)},
\]	
where $K$ runs through the compact subsets of $X$ and $q$ runs through the continuous seminorms on $\E$. 
As lcsc groups are $\sigma$-compact this construction applies to $G$ and its powers and it was shown in \cite[Corollaire 3.5]{blanc} that the complex
\begin{align}\label{eq:L^p-loc-resolution}
0\To \E \overset{\iota}{\To} L^p_{\loc}(G,\E)\overset{d^{0}}{\To}  L^p_{\loc}(G^2,\E)\overset{d^{1}}{\To}  L^p_{\loc}(G^3,\E) \overset{d^{2}}{\To} \cdots,
\end{align}
with $G$-action and coboundary maps given by the natural analogues of \eqref{eq:standard-action} and \eqref{eq:standard-differential}, is also a strengthened, relatively injective resolution.
\subsubsection{Property $H_T$}\label{subsec:cohomological-properties}

Recall that a lcsc group $G$ is said to have Shalom's \emph{property $H_T$} (see \cite{shalom-harmonic-analysis})   if for any unitary, complex Hilbert $G$-module $\H$ (i.e.~$\H$ is a complex Hilbert space endowed with the structure of a continuous $G$-module  and each element in $G$ acts as a unitary operator) the inclusion $\H^G\hookrightarrow \H$ induces an isomorphism $\underline{\Cohom}^1(G,\H^G)\simeq \underline{\Cohom}^1(G,\H)$ or, equivalently, if $\underline{\Cohom}^1(G,\H)=\{0\}$ for every unitary $G$-module $\H$ with no non-trivial fixed-points.  In this spirit, we will say that a group has property $H_T^n$ if $\underline{\Cohom}^n(G,\H)=\{0\}$ for every (complex) unitary $G$-module $\H$ with no non-trivial fixed-points. Since our main focus in the present paper is on groups acting on \emph{real} Hilbert spaces, we remark at this point that  a group has property $H_T^n$  if and only if it has the analogously defined property for real Hilbert spaces; this is easily seen by applying restriction of scalars and complexification, respectively. 

Note that it follows by general disintegration theory and \cite[Theorem 10.1]{blanc} that csc nilpotent Lie groups have property {$H_T^n$ for all $n$, and from this and Mal'cev theory \cite{malcev} it also follows that torsion free, finitely generated nilpotent groups have property $H_T^n$ for all $n$}; see \cite[Theorem 4.1.3]{shalom-harmonic-analysis} for more details. \\ 

\subsubsection{Betti numbers}
The \emph{Betti numbers} of a group $G$ are  defined as $\beta^n(G):=\dim_{\RR}\Cohom^n(G,\RR)$ and we remark that these are finite whenever $G$ is a csc nilpotent Lie group  or a torsion free, discrete nilpotent group. For the latter, note that the classifying space of such groups are finite CW-complexes and the group cohomology agrees with the cohomology of the classifying space, and the former can be seen for instance by passing to Lie algebra cohomology via the van Est theorem \cite{van-est}.
Moreover, when $\beta^n(G)<\infty$, then  $\Cohom^n(G,\RR)$  is automatically Hausdorff \cite[III, Proposition 2.4]{guichardet-book} and hence there is no difference between reduced and ordinary cohomology.

\subsubsection{Cup products}\label{sec:cup}
For a discrete group $\Gamma$, it is well-known that its  cohomology with real coefficients $\Cohom^*(\Gamma, \RR)=\oplus_{n\geq 0} \Cohom^n(\Gamma, \RR)$ becomes a unital, graded-commutative $\RR$-algebra for the so-called cup product, and the same construction also works for lcsc groups,  but since this does not seem properly documented in the existing literature, we recall the construction below. If $G$ is  lcsc groups and  $\xi\in C(G^{n+1}, \RR)$ and $\eta\in C(G^{m+1},\RR)$, one defines their \emph{cup product} $\xi\smile  \eta\in C(G^{n+m+1},\RR)$ as
\[
(\xi\smile \eta)(g_0,\dots, g_{n+m}):=\xi(g_0, \dots, g_n)\eta(g_n,\dots, g_{n+m}).
\]
As the cup product commutes with the (left regular) $G$-action, in the sense that $(g.\xi)\smile (g.\eta)=g.(\xi\smile \eta)$, it descends to a map 
\[
\smile\colon C\left(G^{n+1},\RR\right)^G\times C\left(G^{m+1}, \RR\right)^G\To C\left(G^{n+m+1},\RR\right)^{G}.
\]
Moreover,  the standard differentials \eqref{eq:standard-differential} satisfy a graded Leibniz rule, i.e.~
\begin{align}\label{eq:Leibniz}
d^{n+m}(\xi\smile \eta)=d^n(\xi ) \smile \eta + (-1)^n\xi\smile d^m(\eta),
\end{align}
from which it follows that the cup product passes down to the level of cocycles and that the set of coboundaries is an ideal, so that the cup product descends to a map
\[
\smile\colon \Cohom^n(G,\RR) \times \Cohom^m(G,\RR) \To \Cohom^{n+m}(G,\RR),
\]
which  turns $\Cohom^*(G,\RR)$ into a unital, graded-commutative $\RR$-algebra. 
We will encounter more elaborate cup products in Section \ref{sec:proof-of-B} satisfying natural analogues of \eqref{eq:Leibniz}, so  to make the argument easily available we now give the short proof of  \eqref{eq:Leibniz}, in order to  leave the details later to the reader in good conscience. To prove \eqref{eq:Leibniz}, let cocycles $\xi\in C(G^{n+1}, \RR)$, $\eta\in C(G^{m+1},\RR)$ be given and note that
\begin{align*}
&\scalebox{0.92}[1]{$d^{n+m}(\xi\smile \eta)(g_0\dots g_{n+m+1}) =\displaystyle\sum_{i=0}^{n+m+1}(-1)^{i}(\xi\smile \eta)(g_0\dots \hat{g}_i\dots g_{n+m+1})=$}\\
&\scalebox{0.92}[1]{$=\displaystyle \sum_{i=0}^{n}(-1)^{i}\xi(g_0\dots \hat{g}_i\dots g_{n+1})\eta(g_{n+1}\dots g_{n+m+1}) 
+ \sum_{i=n+1}^{n+m+1}(-1)^{i}\xi(g_0 \dots g_{n})\eta(g_{n} \dots \hat{g}_i\dots g_{n+m+1})$}\\
&\scalebox{0.92}[1]{$=\displaystyle \sum_{i=0}^{n}(-1)^{i}\xi(g_0\dots \hat{g}_i\dots g_{n+1})\eta(g_{n+1}\dots g_{n+m+1}) 
+ \sum_{i=1}^{m+1}(-1)^{i+n}\xi(g_0 \dots g_{n})\eta(g_{n} \dots \hat{g}_{n+i}\dots g_{n+m+1})$}
\end{align*}
On the other hand
\begin{align*}
&\scalebox{0.91}[1]{$\left(d^n(\xi ) \smile \eta + (-1)^n\xi\smile d^m(\eta)\right)(g_0\dots g_{n+m+1})$=}\\
& \scalebox{0.91}[1]{$=\displaystyle\sum_{i=0}^{n+1}(-1)^i\xi(g_0 \dots \hat{g}_i\dots g_{n+1})\eta(g_{n+1}\dots g_{n+m+1}) +(-1)^n\sum_{i=0}^{m+1}(-1)^{i}\xi(g_0\dots g_n)\eta(g_{n}\dots \hat{g}_{n+i}\dots g_{n+m+1})$}\\
&\scalebox{0.91}[1]{$=\displaystyle\sum_{i=0}^{n}(-1)^i\xi(g_0\dots \hat{g}_i\dots g_{n+1})\eta(g_{n+1}\dots g_{n+m+1}) 
+(-1)^n\sum_{i=1}^{m+1}(-1)^{i}\xi(g_0\dots g_n)\eta(g_{n}\dots \hat{g}_{n+i}\dots g_{n+m+1})$},
\end{align*}
where the last equality follows since the last summand in the first sum cancels with the first summand in the second; this proves \eqref{eq:Leibniz}.\\

A direct computation shows that the cup product is continuous with respect to the topology of uniform convergence on compacts at the level of cochains, and since we just saw that the coboundaries constitute an ideal, the same is true for the their closure. Thus, the cup product also descends to a product on $\underline{\Cohom}^*(G,\RR)$.

\subsection{Measure equivalence}
In this section we review the necessary theory concerning (uniform) measure equivalence for locally compact groups; we refer to \cite{BFS-integrable} and \cite{KKR17} for more details on the involved notions.

\begin{defi}[{\cite{BFS-integrable}}]\label{def:me}
  Two  unimodular lcsc groups $G$ and $H$ with Haar measures $\lambda_G$ and $\lambda_H$ are said to be \emph{measure equivalent} if there exist a standard Borel measure $G\times H$-space $(\Omega,\eta)$  and two standard Borel measure spaces $(X,\mu)$ and $(Y,\nu)$ such that: 
  \begin{itemize}
  \item[(i)] both $\mu$ and $\nu$ are finite measures and $\eta$ is non-zero;
  \item[(ii)] there exists an isomorphism of measure $G$-spaces  $i\colon (G\times Y, \lambda_G\times \nu) \To(\Omega,\eta)$, where $\Omega$ is considered a measure $G$-space for the restricted action and $G\times Y$ is considered a measure $G$-space for the action $g.(g',y)=(gg',y)$;

  \item[(iii)] there exists an isomorphism of measure $H$-spaces $j:(H\times X, \lambda_H\times \mu) \to (\Omega,\eta)$, where $\Omega$ is considered a measure $H$-space for the restricted action and $H\times X$ is considered a measure $H$-space for the action $h.(h',x)=(hh',x)$.
  \end{itemize}
  A standard Borel space $(\Omega,\eta)$ with these properties is called a \emph{measure equivalence coupling} between $G$ and $H$, and whenever needed we will specify the additional data by writing  $(\Omega,\eta, X,\mu, Y,\nu, i,j)$. Note also that since the Haar measure is unique up to scaling, the existence of a measure equivalence coupling is independent of the choice of Haar measures on the two groups in question.
\end{defi}
Any measure equivalence coupling gives rise to measure preserving actions $G\curvearrowright (X,\mu)$ and $H\curvearrowright (Y,\nu)$ as well as two $1$-cocycles $\omega_G\colon H\times Y \to G$ and $\omega_H\colon G\times X\to H$. These are defined almost everywhere by the relations
\begin{align*}
i(g\omega_G(h,y)^{-1}, h.y)&=h.i(g,y), \quad \text{for almost all $g\in G$}.\\
j(h\omega_H(g,x)^{-1}, g.x)&=g.j(h,x), \quad \text{for almost all $h\in H$}.
\end{align*}
In the definition of the actions and cocycles above, we are paying little attention to the measure theoretical subtleties, but the reader may find these worked out in detail in \cite[Section 2]{KKR17}. Note that it was was also proven in \cite{KKR17} that one can always obtain a \emph{strict} measure equivalence coupling; i.e.~one in which the maps $i$ and $j$ are Borel isomorphisms and globally equivariant.

\begin{defi}[{\cite{KKR17}}]\label{def:ume}
A strict measure equivalence coupling $(\Omega,\eta, X,\mu, Y,\nu,i,j)$ between uni\-modular, lcsc groups $G$ and $H$ is said to be \emph{uniform} if
\begin{itemize}
\item[(i)] {f}or every compact $C\subset G$ there exists a compact $D\subset H$ such that $j^{-1}\circ i(C\times Y)\subset D\times X${;} 
\item[(ii)] {f}or every compact $D\subset H$ there exists a compact $C\subset G$ such that $i^{-1}\circ j(D\times X)\subset C\times Y$.
\end{itemize}
In this case $G$ and $H$ are said to be \emph{uniformly measure equivalent} (UME), and the properties (i) and (ii) are referred to as the cocycles being \emph{locally bounded}.

\end{defi}
As mentioned already, measure equivalence was originally introduced by Gromov as a measure theoretic analogue of quasi-isometry, and although neither property in general implies the other, it was proven in \cite[Proposition 6.13]{KKR17} that for compactly generated\footnote{Recall that for compactly generated groups coarse equivalence coincides with quasi-isometry}, unimodular, lcsc groups, uniform measure equivalence always implies quasi-isometry, and that  the converse holds under the additional assumption of amenability \cite[Theorem 6.15]{KKR17}. This generalises earlier results for discrete groups by Shalom \cite{shalom-harmonic-analysis} and Sauer \cite{sauer}. \\

If $(\Omega, \eta, X,\mu,Y,\nu, i,j)$ is a strict UME coupling between $G$ and $H$ and $\E$ is a Fr{\'e}chet space then we define $L^p_{\loc}(\Omega,\E)$ as those (equivalence classes modulo equality $\eta$-almost everywhere) of measurable functions $f\colon \Omega\to \E$ such that for every $C\subset G$ compact and every continuous seminorm $q$ on $\E$, one has
\[
\int_{C\times Y} q(f\circ i(g,y))^2 \d\lambda_G(g)\d \nu(y)<\infty,
\]
We topologise $L^2_{\loc}(\Omega, \E)$ via the seminorms $q_C(f):= \sqrt{\int_{C\times Y} q(f\circ i(g,y))^2 \d\lambda_G(g)\d \nu(y)}$ and endow it with the $G\times H$- action
\begin{align}\label{eq:action-on-L2-loc-Omega}
((g,h).f)(t):= (g,h). f((g,h)^{-1}.t), \ t\in \Omega, (g,h)\in G\times H.
\end{align}
We could of course also have defined $L^p_{\loc}(\Omega,\E)$ using the map $j$ instead of $i$, but since $\Omega$ is uniform this gives rise to exactly the same space. To see this, take $f$, $C$ and $q$ as above. By local boundedness of the cocycles, we can find $D\subset H$ compact, such that $j^{-1}\circ i(C\times Y)\subset D\times X$ and consequently
\begin{align*}
\int_{C\times Y} q(f\circ i(g,y))^2 \d\lambda_G(g)\d \nu(y)&=\int_{C\times Y} q(f\circ j\circ j^{-1}\circ i(g,y))^2 \d\lambda_G(g)\d \nu(y)
\\
&=\int_{j^{-1}\circ i(C\times Y)} q(f\circ j(h,x))^2 \d\lambda_H(h)\d \mu(x)
\\
&\leq\int_{D\times Y} q(f\circ j(h,x))^2 \d\lambda_H(h)\d \mu(x).
\end{align*}
This shows that each of the seminorms defining the topology via the isomorphism $i$, is dominated by one of the seminorms defined via $j$, and by symmetry the converse is true as well, and the two families therefore generate the same topology on $L^2_{\loc}(\Omega,\E)$. 
\begin{lem}
With the action defined by \eqref{eq:action-on-L2-loc-Omega} the space $L^2_{\loc}(\Omega, \E)$ becomes a Fr{\'e}chet $G\times H$-module. 
\end{lem}
\begin{proof}
Since $Y$ is standard Borel we may equip it with a compact metrizable topology generating the $\sigma$-algebra \cite[Theorem 15.6]{kechris-book}, and the space $L^2_{\loc}(\Omega,\E)$ then directly identifies with $L^2_{\loc}(G\times Y, \E)$ which is Fr{\'e}chet by \cite[D.2]{guichardet-book}.  It suffices to show that both groups act continuously, and by symmetry it is enough to treat the $G$-action. To this end, note that \cite[D.2.2 (vii)]{guichardet-book} gives an isomorphism of $G$-modules
\[
L^2_\loc(\Omega, \E)\simeq L^2_{\loc}(G, L^2(Y,\E)),
\]
where the action on the right hand side is given by $(g.\xi)(g')(y)=g.(\xi(g^{-1}g')(y))$, so by \cite[D.3.2]{guichardet-book} it suffices to show that the pointwise $G$-action on $L^2(Y,\E)$ is continuous. To see this, note that $C(Y,\E)$ is dense in $L^2(Y,\E)$, so by \cite[Lemme D.8 (ii)]{guichardet-book}, it is enough to show that the action is equicontinuous over compact sets   and pointwise continuous on elements from $C(Y,\E)$. To see the former, let $K\subset G$ be compact and $q$ be a continuous seminorm on $\E$. Then since $G\curvearrowright \E$ is continuous, there exists a continuous seminorm $q'$ on $\E$ such that for all $g\in K$ and $x\in \E$ we have $q(g.x)\leq q'(x)$. Hence, for $g\in K$ and $\xi\in L^2(Y, \E)$  we also have  that
\[
\int_Y q((g.\xi)(y))^2\d \nu(y)= \int_Y q(g.\xi(y))^2\d \nu(y) \leq \int_Y q'(\xi(y))^2\d \nu(y)
\]
showing equicontinuity over compact sets. To see that the action is pointwise continuous on $\xi\in C(Y,\E)$, simply note that if $g_n\to g$ in $G$  and $q$ is a continuous seminorm on $\E$ then 
\begin{align*}
\int_Y(q(g_n.\xi-g\xi)(y))^2 \d \nu(y)& = \int_Y(q(g_n.\xi(y)-g.\xi(y))^2 \d \nu(y)\\
&\leq \nu(Y) \sup_{x\in \xi(Y)} q(g_n.x-g.x)^2 \underset{n\to\infty}{\To} 0,
\end{align*}
where the convergence follows from compactness of $\xi(Y)\subset \E$ and \cite[Lemme D.8 (iii)]{guichardet-book}.
\end{proof}

\section{Proof of Theorem \ref{mainthm:UME-induction}}
The aim of the current section is to prove Theorem  \ref{mainthm:UME-induction}, so we fix uniformly measure equivalent, unimodular lcsc groups $G$ and $H$  and a strict, uniform measure equivalence coupling $(\Omega, \eta,X\mu, Y,\nu, i,j)$ between them, as well as Fr{\'e}chet $G\times H$-module $\E$. Consider the space $L^2_{\loc}(G^n, L^2_{\loc}(\Omega,\E))$ endowed with  $G\times H$-action 
\[
((g,h).\xi)(g_1,\dots, g_n)(t):=(g,h)[\xi(g^{-1}g_1,\dots, g^{-1}g_n)((g,h)^{-1}.t)], 
\]
for $(g,h)\in G\times H, g_1,\dots, g_n\in G, t\in \Omega.$ 

\begin{lem}

The space $L^2_{\loc}(G^n, L^2_{\loc}(\Omega,\E))$  is a relatively injective $G\times H$-module and the complex $R_G$ defined as
\begin{align}\label{eq:R_G-complex}
0\To L^2_{\loc}(\Omega, \E) \To L^2_{\loc}\left(G, L^2_{\loc}(\Omega, \E)\right) \To L^2_{\loc}\left(G^2, L^2_{\loc}(\Omega, \E)\right) \To \cdots
\end{align}
constitutes a strengthened, relatively injective resolution of the $G\times H$-module  $L^2_{\loc}(\Omega, \E)$ when endowed with the standard homogeneous differentials, given by the obvious {extension} of the formula \eqref{eq:standard-differential}. 
\end{lem}
\begin{proof}
Throughout the proof we identify $\Omega$ with $H\times X$ through the map $j$ defining the $L^2_{\loc}$-structure on $\Omega$; recall that under this identification the $G\times H$-action on 
\[
L^2_{\loc}\left(G^{n+1}, L^2_{\loc}(\Omega,\E)\right)=L^2_{\loc}\left(G^{n+1}, L^2_{\loc}(H\times X,\E)\right)
\]
 is given by the formula
\[
((g,h).\xi)(g_0,\dots, g_n)(h',x)=(g,h).\left[\xi\left(g^{-1}g_0,\dots, g^{-1}g_n \right)(h^{-1}h'\omega_H(g^{-1},x  )^{-1}, g^{-1}.x)\right].
\]
Since $X$ is standard Borel, \cite[Theorem 15.6]{kechris-book} ensures that we can find a compact, metrizable topology on it whose open sets generate the Borel structure, and we may therefore consider the Fr{\'e}chet space $L^2_{\loc}(G^n \times X,\E)$. 
Arguing as in \cite[n$^{\circ}$ D.3.2.]{guichardet-book}, we obtain that this is a Fr{\'e}chet $G\times H$-module when endowed with the $G\times H$-action
\[
((g,h).\xi)(g_1,\dots, g_n,x) :=(g,h).\xi(g^{-1}g_1, \dots, g^{-1}g_n, g^{-1}.x),
\]
and by \cite[Th{\'e}or{\`e}me 3.4]{blanc}, we therefore have that
\[
L^2_{\loc}\left (G\times H, L^2_{\loc}(G^n \times X,\E)\right),
\]
with the standard $G\times H$-action, is a relatively injective Fr{\'e}chet $G\times H$-module. A routine calculation now shows that the map
\begin{align*}
\alpha\colon L^2_{\loc}\left(G^{n+1}, L^2_{\loc}(H\times X,\E)\right) &\To L^2_{\loc}(G\times H, L^2_{\loc}(G^n \times X,\E))\\
\alpha(\xi)(g,h)(g_1,\dots, g_n,x)&:=\xi(g_1,\dots, g_n,g)(h\omega_H(g^{-1},x),x),
\end{align*}
is an isomorphism of Fr{\'e}chet spaces intertwining the $G\times H$-actions, and hence it follows that also $L^2_{\loc}\left(G^{n+1}, L^2_{\loc}(\Omega,\E)\right)$ is  relatively injective; we will here only show equivariance of $\alpha$, leaving the verification of  invertibility and continuity to the reader: Acting with $(\bar{g}, \bar{h})\in G\times H$ before and after applying $\alpha$ gives
\begin{align*}
\alpha((\bar{g},\bar{h}).\xi)(g,h)(g_1,\dots,g_{n-1},x) = ((\bar{g},\bar{h}).\xi)(g_1,\dots, g_n,g)(h\omega_H(g^{-1},x),x)
\\
=(\bar{g},\bar{h}).\left[ \xi(\bar{g}^{-1}g_1,\dots , \bar{g}^{-1}g_{n-1},\bar{g}^{-1}g)(\bar{h}^{-1}h\omega_H(g^{-1},x)\omega_H(\bar{g}^{-1},x)^{-1},\bar{g}^{-1}x)\right],
\end{align*}
and
\begin{align*}
((\bar{g},\bar{h}).\alpha(\xi))(g,h)(g_1,\dots,g_{n-1},x)=(\bar{g},\bar{h}).\left[\alpha(\xi)(\bar{g}^{-1}g,\bar{h}^{-1}h)(\bar{g}^{-1}g_1,\dots,\bar{g}^{-1}g_{n-1},\bar{g}^{-1}x)\right]
\\
=(\bar{g},\bar{h})\left[\xi(\bar{g}^{-1}g_1,\dots,\bar{g}^{-1}g_{n-1},\bar{g}^{-1}g)(\bar{h}^{-1}h\omega_H(g^{-1}\bar{g},\bar{g}^{-1}x),\bar{g}^{-1}x)\right],
\end{align*}
respectively. The result now follows, since the cocycle identity implies
$$\omega_H\left(g^{-1}\bar{g},\bar{g}^{-1}x\right)\omega_H\left(\bar{g}^{-1},x\right)=\omega_H\left(g^{-1},x\right).$$ 
Since the complex $R_G$ is simply the standard $L^2_{\loc}$-resolution of $L^2_{\loc}(\Omega, \E)$ considered only as a $G$-module, it is clearly strengthened, and hence the proof is complete. 
\end{proof}

\begin{proof}[Proof of Theorem \ref{mainthm:UME-induction}]
Since the roles of $G$ and $H$ are symmetric in the module structure on $L^2_{\loc}(\Omega,\E)$ one may construct a strengthened, relatively injective resolution $R_H$ of $L^2_{\loc}(\Omega,\E)$ analogous to \eqref{eq:R_G-complex}, whose degree $n$ term is given by $L^2_{\loc}(H^{n+1},L^2_{\loc}(\Omega,\E))$. Thus, both $R_G$ and $R_H$ compute $\Cohom^n(G\times H,L^2_{\loc}(\Omega, \E))$ after passing to $G\times H$-invariants and cohomology. However,
\[
L^2_{\loc}\left(G^{n+1},L^2_{\loc}(\Omega,\E)\right)^{G\times H}=\left(L^2_{\loc}(G^{n+1},L^2_{\loc}(\Omega,\E))^H \right)^G=\left(L^2_{\loc}\left(G^{n+1},L^2_{\loc}(\Omega,\E)^H\right) \right)^G,
\]
where the last equality is due to the fact that $H$ acts trivially in the $G^{n+1}$-direction. From this we see that passing to $G\times H$-invariants in $R_G$ is exactly the same as passing to $G$-invariants in the $L^2_{\loc}$-resolution \eqref{eq:L^p-loc-resolution} of the $G$-module $L^2_{\loc}(\Omega,\E)^{H}$, and hence we obtain a topological isomorphism 
\[
\Cohom^n\left(G, L^2_{\loc}(\Omega, \E)^H\right) \simeq\Cohom^n(G\times H, L^2_{\loc}(\Omega,\E)).
\]
Replacing $R_G$ with $R_H$, a symmetric argument yields  that
\[
\Cohom^n\left(H, L^2_{\loc}(\Omega, \E)^G\right) \simeq\Cohom^n(G\times H, L^2_{\loc}(\Omega,\E)),
\]
and the desired isomorphism follows. \\
If $\E=\RR$ with trivial $G\times H$-action, then by \cite[D.2.2 (vii) \& Lemme D.9]{guichardet-book}  we have an isomorphism of $H$-modules
\[
L^2_{\loc}(\Omega, \RR)^G \simeq L^2_{\loc}(G\times Y,\RR)^G \simeq L^2_{\loc}(G,L^2(Y))^G \simeq L^2(Y),
\]
 and, similarly, an isomorphism of $G$-modules $L^2_{\loc}(\Omega, \E)^H\simeq L^2(X)$;  hence the last part of the statement follows from the first.
\end{proof}

\begin{rem}\label{rem:cmplx-rem-1}
Theorem \ref{mainthm:UME-induction} is stated for real Fr{\'e}chet spaces and real cohomology, but as seen from the proof just given, the analogous statement over the complex numbers also holds true with verbatim the same proof.
\end{rem}

\begin{rem}\label{rem:chi-map}
The proof just given does not, directly, provide a concrete map realizing the isomorphism $\Cohom^n\left(G, L^2_{\loc}(\Omega, \E)^H\right)\simeq \Cohom^n\left(H, L^2_{\loc}(\Omega, \E)^G\right)$, but for  applications, e.g.~our Theorem \ref{mainthm:ring-iso}, having a concrete map is very useful, and we shall therefore now describe one such a map.
By general relative homological algebra \cite[III, Corollaire 1.1]{guichardet-book} , any morphism $\chi$ of complexes of $G\times H$-modules from $R_G$ to $R_H$ which lifts the identity on $L^2_{\loc}(\Omega,\E)$ will induce a (topological) isomorphism $\Cohom^*\big(R_G^{G\times H}, d^n\restriction\big) \simeq \Cohom^*\big(R_H^{G\times H}, d^n\restriction\big)$.  We now define such a map 
\begin{align*}
\chi^n\colon L^2_{\loc}\left(G^{n+1}, L^2_{\loc}(\Omega, \E)\right) &\To L^2_{\loc}\left(H^{n+1}, L^2_{\loc}(\Omega, \E)\right)
\end{align*}
by setting 
\[
\chi^n(\xi)(h_0,\dots, h_n)(t):= \xi\left(\pi_G\circ i^{-1}(h_0^{-1}.t), \dots, \pi_G\circ i^{-1}(h_n^{-1}.t   )\right)(t),
\]
where $\pi_G\colon G\times Y \to G$ denotes the projection onto the first factor. Identifying $\Omega$ with $G\times Y$ via $i$ the map takes the form
\[
%\chi^n(\xi)(h_0,\dots, h_n)(g,y)=\xi\left(\omega_G(h_0,y)^{-1}, \dots, \omega_G(h_n,y)^{-1}\right)(g,y).
\chi^n(\xi)(h_0,\dots, h_n)(g,y)=\xi\left(g\omega_G(h_0^{-1},y)^{-1}, \dots, g\omega_G(h_n^{-1},y)^{-1}\right)(g,y).
\]

Note that the map $\chi^n$ does indeed take values in  $L^2_{\loc}(H^{n+1}, L^2_{\loc}(\Omega, \E))$ since $\Omega$ is uniform such that the cocycles are locally bounded, and from this it also follows that $\chi^n$ is continuous. It is straightforward to see that $\chi$ is a chain map lifting the identity on $L^2_{\loc}(\Omega, \E)$ and from the cocycle identity it follows that $\chi^n$ is $G\times H$-equivariant.
\end{rem}

\section{Proof of Theorem \ref{mainthm:ring-iso}}\label{sec:proof-of-B}
In this section we prove Theorem \ref{mainthm:ring-iso}. We thus
assume that $G$ and $H$ are uniformly measure equivalent, unimodular, lcsc groups satisfying property {$H_T^n$ for all $n$}, and fix a strict, uniform measure equivalence coupling $(\Omega, \eta, X,\mu, Y,\nu, i,j)$. By \cite[Proposition 2.13]{KKR17}, and its proof, one may change the measures $\eta,\mu$ and $\nu$ into ones that are ergodic for the $G\times H$-, $G$- and $H$-actions, respectively, so we may, and shall, assume that the original measures are ergodic. Furthermore, by rescaling the measures involved, we may also assume that $\nu$ is a probability measure. We therefore obtain isomorphisms
\begin{align*}
\underline{\Cohom}^n(G,\RR)  &=  \underline{\Cohom}^n\left(G, L^2(X)^G\right) \tag{ergodicity}\\
&\simeq   \underline{\Cohom}^n\left(G, L^2(X)\right) \tag{property {$H_T^n$}}\\
&\simeq   \underline{\Cohom}^n\left(G, L_{\loc}^2(\Omega)^H\right) \\
&\simeq \underline{\Cohom}^n\left(H, L_{\loc}^2(\Omega)^G\right) \tag{Theorem \ref{mainthm:UME-induction}}\\
&\simeq \underline{\Cohom}^n\left(H, L^2(Y)\right)\\
&\simeq \underline{\Cohom}^n\left(H, \RR\right) \tag{ergodicity and property {$H_T^n$}}
\end{align*}
Chasing through the isomorphisms above, using the explicit isomorphism $\chi^n$ from Remark \ref{rem:chi-map}, shows that
\begin{align*}
\kappa^n\colon L^2_{\loc}(G^{n+1}, \RR) &\To L^2_{\loc}(H^{n+1},\RR),
\end{align*}
given by $\kappa^n(\xi)(h_0,\dots, h_n)=\int_{Y} \xi\left(\omega_G(h_0^{-1},y)^{-1},\dots, \omega_G(h_n^{-1},y)^{-1}  \right)\d \nu(y)$ at the cochain level induces the isomorphism $\underline{\Cohom}^n(G,\RR)\simeq \underline{\Cohom}^n(H,\RR)$ and the aim is now to prove that this isomorphism preserves cup products. To this end, we need an auxiliary complex defined as follows:
\begin{defi}\label{def:D^n-complex}
Denote by $D^n$  the subspace of $ L^2_{\loc}(H^{n+1}, L^2(Y)) $ consisting of the classes (modulo equality almost everywhere) of functions $\xi\in \mathcal{L}^2_{\loc}(H^{n+1}, L^2(Y))$ such that:
\begin{enumerate}
\item[(i)] for almost all $(h_0,\dots, h_n)\in H^{n+1}$: $\xi(h_0,\dots, h_n)\in L^\infty(Y)\subset L^2(Y)$, and
\item[(ii)]  for all $C\subset H^{n+1}$ compact:  $\text{ess sup} \{\| \xi(h_0,\dots, h_n)\|_{\infty} : (h_0,\dots, h_n)\in C \} <\infty. $
\end{enumerate}
\end{defi}
Since $L^\infty(Y)\subset L^2(Y)$ is an $H$-invariant subspace, $D^n$ becomes a (non-complete) $H$-submodule of $L^2_{\loc}(H^{n+1},L^2(Y))$. Moreover, it is easily seen that $d^n(D^n)\subseteq D^{n+1}$  and hence 
\[
D: \ (D^0)^H \overset{d^0\restriction}{\To} (D^1)^H \overset{d^1\restriction}{\To} (D^2)^H \overset{d^2\restriction}{\To} \cdots
\]
is a subcomplex of the standard $L^2_\loc$-complex computing $\Cohom^*(H,L^2(Y))$.
\begin{defi}

For $\alpha \in D^n, \beta\in D^m$ and $\xi \in L^2_{\loc}(H^{n+1},L^2(Y))$ we define $\alpha\smile \xi, \xi \smile \alpha \in L^2_{\loc}(H^{m+n+1},L^2(Y))$ and $\alpha\smile \beta \in D^{n+m}$ by
\begin{align}
\alpha\smile \xi(h_0,\dots, h_{n+m})&:=\alpha(h_0,\dots, h_n)\xi(h_n,\dots, h_{n+m}) \ \label{eq:cup-left}\\
\xi \smile \alpha(h_0,\dots, h_{n+m})&:=\xi(h_0,\dots, h_m)\alpha(h_m,\dots, h_{n+m}) \label{eq:cup-right} \\
\alpha\smile \beta(h_0,\dots, h_{n+m})&:=\alpha(h_0,\dots, h_n)\beta(h_n,\dots, h_{n+m})\label{eq:cup-D}
\end{align}
where the products on the right hand side are the pointwise products between functions in $L^\infty(Y)$ and $L^2(Y)$. 
\end{defi}

\begin{lem}\label{lem:cohomology-bimodule-structure}
The product defined by \eqref{eq:cup-D} descends to a product on $\Cohom^*(D)$ turning it into a unital $\RR$-algebra. Similarly, the products \eqref{eq:cup-left} and \eqref{eq:cup-right} descend to the level of cohomology and reduced cohomology turning $\Cohom^*(H,L^2(Y))$ and $\underline{\Cohom}^*(H,L^2(Y))$ into bimodules for $\Cohom^*(D)$.   
\end{lem}
\begin{proof}

For $f_1\in L^\infty(Y)$, $f_2\in L^2(Y)$ and $h\in H$ one has
\[
h.(f_1\cdot f_2)=(h.f_1)\cdot (h.f_2) \ \text{ and } h.(f_2\cdot f_1)=(h.f_2)\cdot (h.f_1) 
\]
(the products being defined pointwise) and from this it follows that the cup product of $H$-invariant elements is again $H$-invariant, so that $\smile$ restricts to the level of fixed points for $H$. Moreover, by repeating the proof of  \eqref{eq:Leibniz} mutatis mutandis, the cup products  \eqref{eq:cup-left}, \eqref{eq:cup-right}  and \eqref{eq:cup-D}, are seen to satisfy the obvious versions of the Leibniz rule, from which it follows that  $\Cohom^*(D)$ becomes a (graded) $\RR$-algebra for which $\Cohom^*(H,L^2(Y))$ is a (graded) bimodule. We will omit the details of the latter, and instead show that  $\underline{\Cohom}^*(H,L^2(Y))$ becomes a $\Cohom^*(D)$-bimodule, since this proof contains the other as a special case. We need to show that if $\alpha, \alpha'\in D^n \cap \ker(d^n)$ and $\xi, \xi'\in L^2_{\loc}(H^{m+1},L^2(Y))\cap \ker(d^m)$ satisfy that $\alpha-\alpha'\in d^{n-1}(D^{n-1})$ and $\xi-\xi'\in \cl\left(d^{m-1}(L^2_{\loc}(H^{m},L^2(Y))\right)$ then $\alpha \smile \xi-\alpha'\smile \xi'\in \cl\left(d^{n+m-1}(L^2_{\loc}(H^{n+m},L^2(Y))\right) $. Write $\alpha-\alpha'=d^{n-1}(\beta)$ and note that, by the Leibniz rule, we have

\begin{align*}
\alpha\smile \xi - \alpha'\smile \xi' &= (\alpha-\alpha')\smile \xi- \alpha'\smile(\xi-\xi')\\
&=d^{n-1}(\beta) \smile \xi- \alpha'\smile(\xi-\xi')\\
&= d^{n+m-1}(\beta \smile \xi)-(-1)^{n-1}\beta \smile d^m(\xi) - \alpha'\smile(\xi-\xi')\\
&= d^{n+m-1}(\beta \smile \xi) - \alpha'\smile(\xi-\xi').
\end{align*}
By assumption, there exist $\eta_k\in L^2_{\loc}(H^{m},L^2(Y))$ such that $\lim_k d^{m-1}(\eta_k)=\xi-\xi'$, so to finish the proof it suffices to show that the cup product is pointwise continuous in the second variable; i.e.~that if $\zeta_k\to_k 0 $ in $L^2_{\loc}(H^{m+1},L^2(Y))$ then $\alpha\smile \zeta_k \to_k 0$ in $L^2_{\loc}(H^{m+n+1},L^2(Y))$. To see this, let a compact $K\subset H^{n+m+1}$  be given. Upon passing to a bigger compact set, we may assume that $K=\prod_{i=0}^{n+m}K_i$ for some compact subsets $K_i\subset H$ and we now get
\begin{align*}
& \int_{K} \|\alpha \smile \zeta_k\|_2 \d \lambda_H^{\tens (n+m+1)}= \int_{\prod_{i=0}^{n+m}K_i} \|\alpha(h_0,\dots, h_n)\zeta_k(h_n,\dots, h_{n+m})\| \d h_0 \cdots \d h_{n+m}\\
&\leq \sup_{(h_0,\dots, h_n)\in \prod_{i=0}^{n-1}K_i} \|\alpha(h_0,\dots, h_n)\|_\infty\prod_{i=0}^n \lambda_H(K_i)  \int_{\prod_{i=n}^{n+m}K_i} \|\zeta_k(h_n,\dots, h_{n+m})\| \d h_n \cdots \d h_{n+m}.
\end{align*}
By definition of $D^n $ we have 
\[
 \sup_{(h_0,\dots, h_n)\in \prod_{i=0}^nK_i} \|\alpha(h_0,\dots, h_n)\|_\infty< \infty,
 \]
  and by assumption 
  \[
   \int_{\prod_{i=n}^{n+m}K_i} \|\zeta_k(h_n,\dots, h_{n+m})\| \d h_n \cdots \d h_{n+m} \to_k 0.  \qedhere
   \]
\end{proof}

\begin{rem}
Using the usual contracting homotopy (cf.~\cite[III, Proposition 1.4]{guichardet-book}) for the resolution $(L^2_{\loc}(H^{n+1},L^2(Y)),d^n)_{n\in \NN_0}$,  one easily checks that $(D^n, d^n)_{n\in \NN_0}$ is a strengthened resolution of the $H$-submodule $L^\infty(Y)\subset L^2(Y)$, but it seems less clear whether the modules $D^n$ are relatively injective or not. If this were the case, then $\Cohom^*(D)$ would of course be nothing but $\Cohom^*(H,L^\infty(Y))$, where $L^\infty(Y)$ is considered as an $H$-module with respect to the $2$-norm.
\end{rem}

With the lemmas above at our disposal, we can now prove Theorem \ref{mainthm:ring-iso} following the strategy in \cite{sauer} almost verbatim. Recall that we have fixed a strict, ergodic UME coupling $(\Omega, \eta, X,\mu,Y,\nu, i,j)$ and normalised the measures so that $\nu$ has total mass 1.

\begin{proof}[{Proof of Theorem \ref{mainthm:ring-iso}}]
Denote by $L^2_0(Y)$ the functions in $L^2(Y)$ that integrate to zero. 
The Hilbert $H$-module $L^2(Y)$ then splits as an (orthogonal) direct sum $L^2(Y)=L^2_0(Y) \oplus \RR$, via the map $\xi \mapsto \left(\xi-(\int_Y\xi\d\nu)1_Y, \int_Y\xi\d \nu \right)$. From this it follows that we get a decomposition of $H$-complexes
\[
\Phi\colon L^2_{\loc}(H^{n+1},L^2(Y)) \simeq L^2_{\loc}(H^{n+1},L^2_0(Y)) \oplus L^2_{\loc}(H^{n+1},\RR),
\]
by simply composing a function $L^2_{\loc}(H^{n+1},L^2(Y))$ with the two projections $p_1\colon L^2(Y) \to L_0^2(Y)$ and $p_{2}\colon L^2(Y)\to \RR$.
Similarly, denote by $D_0^n$ the subcomplex of $D^n$ consisting of the (classes of) those functions that integrate to zero almost everywhere and by $D^n_{\RR}$ the 
(classes of functions) that, almost everywhere, are almost everywhere constant on $Y$, and note that the decomposition $\Phi$ maps $D^n$ onto $D^n_0 \oplus D^n_{\RR}$.
To be extremely precise, $D^n_0$ consists of classes of those $\xi \in \mathcal{L}^2_{\loc}(H^{n+1}, L^2(Y))$ such that $[\xi]\in D^n$ and such that for almost all $h_0,\dots, h_n $ we have $\int_Y\xi(h_0,\dots, h_{n})d\nu=0$ and $D^n_\RR$ are the classes of those functions  $\xi \in \mathcal{L}^2_{\loc}(H^{n+1}, L^2(Y))$ such that $[\xi]\in D^n$ and such that for almost all $h_0,\dots, h_n $ we have $\xi(h_0,\dots, h_n)=(\int_Y\xi(h_0,\dots, h_{n})d\nu)1_Y $ (the latter equation in $L^\infty(Y)\subset L^2(Y)$
We therefore obtain splittings at the level of cohomology

\begin{align*}
 \Cohom^n(D)&\simeq \Cohom^n(D_0) \oplus\Cohom^n(D_{\RR})\\
\Cohom^n(H,L^2(Y))&\simeq \Cohom^n(H,L_0^2(Y)) \oplus \Cohom^n(H,\RR) \\
\underline{\Cohom}^n(H,L^2(Y))&\simeq \underline{\Cohom}^n(H,L_0^2(Y)) \oplus \underline{\Cohom}^n(H,\RR) 
\end{align*}
respecting the natural downward maps induced by the corresponding inclusions at the level of complexes. 
As shown in the first paragraph of the present section, we have a continuous  linear map
\[
I^*\colon \Cohom^*(G,\RR) {\To} {\Cohom}^*(H,L^2(Y)),
\]
which descends to an isomorphism after the passage to reduced cohomology, and maps  the (class of a) continuous $n$-cocycle $\xi\in C(G^{n+1}, \RR)^G\cap \ker(d^n)$ to the (class of) the cocycle   $I(\xi)\in L^2_{\loc}(H^{n+1}, L^2(Y))$ given by
\[
I(\xi)(h_0,\dots, h_n)(y)= \xi\left(\omega_G(h_0^{-1},y)^{-1},\dots, \omega_G(h_n^{-1},y)^{-1}  \right).
\]
Moreover, since $\Omega$ is uniform, the cocycle $\omega_G$ is locally bounded and from this it follows that $I(\xi)\in D^n$, where $D^n$ is the submodule of $L^2_{\loc}(H^{n+1}, L^2(Y))$ described in Definition \ref{def:D^n-complex}. Hence, $I^*$ factorises as
\[
\Cohom^*(G,\RR) \overset{I^*_0}{\To} \Cohom^*(D) \overset{\iota^*}{\To} \Cohom^*(H, L^2(Y)), 
\]
where $\iota$ is  the map induced by the inclusion $D^n \subset L^2_{\loc}(H^{n+1},L^2(Y))$. Denoting by $\underline{\iota}^*$ the map obtained by composing $\iota^*\colon \Cohom^*(D)\To \Cohom^*(H,L^2(Y))$ with the natural projection $ \Cohom^*(H,L^2(Y))\to \underline{\Cohom}^*(H,L^2(Y))$, we obtain a sequence of maps as follows:
\[
\xymatrix{
&&&\\
\Cohom^*(G,\RR) \ar[r]^{I_0^*}& \Cohom^*(D)   \ar@/^2.0pc/[rr]^{p^*}  \ar[r]^<<<<<{\underline{\iota}^*}& \underline{\Cohom}^*(H,L^2(Y)) \ar[r]^{p_2^*}_{\sim} & \underline{\Cohom}^*(H,\RR) %\\
%&&&
}
\]

Here $p_2^*$ is the  map induced at the level of reduced cohomology by the map $L^2(Y)\to \RR$ given by integration against $\nu$ and $p^*$ is simply defined as $p_2^*\circ \underline{\iota}^*$. Now, $I_0^*$ is multiplicative, since already at the level of cochains we have $I_0(\xi \smile \eta)=I_0(\xi)\smile I_0(\eta)$ for $\xi\in C(G^{n+1}, \RR)$ and $\eta\in C(G^{m+1}, \RR)$, which is seen by a direct computation. 

We now aim to show that also $p^*=p_2^*\circ \underline{\iota}^*$ is multiplicative. To this end, first notice the map $\underline{\iota}^*$   is a bimodule map with respect to the $\Cohom^*(D)$-bimodule structure on $\Cohom^*(D)$ given by left/right multiplication and the $\Cohom^*(D)$-bimodule structure on $\underline{\Cohom}^n(H,L^2(Y))$ described in Lemma \ref{lem:cohomology-bimodule-structure}, and its kernel is therefore a two-sided ideal. Moreover, since $p_2^*$ is an isomorphism, we have $\ker(\underline{\iota}^*)=\ker(p^*)$, and from this we now obtain that $p^*$ is multiplicative as follows: given  $[\xi]\in \Cohom^n(D)$ and $[\eta]\in \Cohom^m(D)$, they decompose as $[\xi]=[\xi_0] + [\xi_1]$ and  $[\eta]=[\eta_0] + [\eta_1]$ for cocycles $\xi_0\in D^n_{0}, \xi_1\in D^n_{\RR}$, $\eta_0\in D^m_0$ and  $\eta_1\in D^m_{\RR}$. By definition, $\underline{\iota}^*([\xi_0]) =\underline{\iota}^*([\eta_0])=0$ and since $\ker(\underline{\iota}^*)$ is an ideal we obtain 

\[
\underline{\iota}^*([\xi]\smile [\eta])=\underline{\iota}^*\left([\xi_0]\smile [\eta_0] + [\xi_0]\smile [\eta_1] +[\xi_1]\smile [\eta_0] +[\xi_1]\smile [\eta_1]\right)= \underline{\iota}^*([\xi_1]\smile [\eta_1])
\]
Since $\xi_1$ and $\eta_1$ take values in essentially constant functions on $Y$ we obtain
\begin{align*}
p_2^*\circ \underline{\iota}^*  \left([\xi_1]\smile [\eta_1] \right)(h_0,\dots, h_{n+m+1}) &= \int_{Y}( \xi_1\smile \eta_1)(h_0,\dots, h_{n+m+1})\d \nu\\
&= \int_Y\xi_1(h_0,\dots, h_n)\eta_1(h_n,\dots, h_{n+m+1})\d\nu\\
&=  \int_Y\xi_1(h_0,\dots, h_n) \d\nu \int_Y \eta_1(h_n,\dots, h_{n+m+1})\d\nu\\
&=p_2^*\circ \underline{\iota}^* \left([\xi_1])\smile p_2^*\circ \underline{\iota}^* ([\eta_1] \right)(h_0,\dots, h_{n+m+1})
\end{align*}
The composition $p^* \circ I_0^*\colon \Cohom^*(G,\RR)\to \underline{\Cohom}^*(H,\RR)$ is therefore multiplicative and since we already argued, in the beginning of this section, that this map descends to an isomorphism $\underline{\Cohom}^*(G,\RR)\simeq \underline{\Cohom}^*(H,\RR)$ it follows that it too is multiplicative.
\end{proof}

\begin{rem}\label{rem:cmplx-rem-2}
Of course one could also define a cup product on $\Cohom^*(G,\CC)$ and the arguments above obviously generalise to this setting so that the analogue of Theorem \ref{mainthm:ring-iso} over the complex numbers holds true as well.
\end{rem}

\section{An application}\label{sec:application}
It is a well known open problem to classify connected simply connected (csc) nilpotent Lie groups up to quasi-isometry, and conjecturally quasi-isometry actually  coincides with isomorphism on this class of groups \cite{yves-qi-classification}. Although  isomorphism classification is wide open in general, 
csc nilpotent Lie groups of dimension at most 7 are completely classified (up to isomorphism) by means of the corresponding classification of nilpotent real Lie algebras. This classification is the work by many hands, and we will here focus on Gong's thesis \cite{gong} which provides the first complete classification of all 7-dimensional, real, nilpotent Lie algebras ({It should, however, be noted that \cite{De-Graaf} has pointed out that this classification is not complete in the 6-dimensional case)}. Regarding the quasi-isometry problem mentioned above, there are basically two main results supporting the conjecture.\\
Firstly, by Pansu's celebrated paper \cite{pansu}, if $G$ and $H$ are quasi-isometric csc nilpotent Lie groups with Lie algebras $\mathfrak{g}$ and $\mathfrak{h}$, respectively, then their associated Carnot Lie algebras $\mathfrak{Car}(\mathfrak{g})$ and $\mathfrak{Car}(\mathfrak{g})$ are isomorphic. Recall that if $\mathfrak{g}$ is a step $c$ nilpotent Lie algebra with  lower central series $\mathfrak{g}=\mathfrak{g}_{1}\geq \mathfrak{g}_{2}\geq \dots \geq \mathfrak{g}_{c}\geq \{ 0\}$, then its \emph{Carnot algebra} is defined  as
\[
\mathfrak{Car}(\mathfrak{g}):=\bigoplus_{i=1}^{c} \mathfrak{g}_{i}/\mathfrak{g}_{i+1},
\]
with Lie bracket given, for $\bar{\xi}\in \mathfrak{g}_{i}/\mathfrak{g}_{i+1}$ and $\bar{\eta}\in \mathfrak{g}_{j}/\mathfrak{g}_{j+1}$, by $[\bar{\xi}, \bar{\eta}]:=\overline{[\xi,\eta]}\in \mathfrak{g}_{i+j}/\mathfrak{g}_{i+j+1}$. The Lie algebra $\mathfrak{g}$ is said to be \emph{Carnot}, if $\mathfrak{g}\simeq \mathfrak{Car}(\mathfrak{g})$, and from Pansu's theorem it therefore follows that if the csc nilpotent Lie groups associated with two Carnot, nilpotent Lie algebras are quasi-isometric, then they are in fact isomorphic.\\ 
Secondly, if $G$ and $H$ both contain lattices, i.e.~admit rational structures \cite{malcev}, then by  \cite[Theorem 5.1]{sauer} and \cite[Theorem 1]{katsumi}  one may conclude that $\Cohom^*(G,\RR)$ and $\Cohom^*(H,\RR)$ are isomorphic as graded-commutative $\RR$-algebras, and this also prevents many low dimensional csc nilpotent Lie groups from being quasi-isometric; see \cite{yves-qi-classification} for the state of the art in this direction. \\
 Our Corollary \ref{mainthm:cor} generalises the latter result (see Section \ref{subsec:cohomological-properties} for this) and in the following sections we give examples showing how it can be used to distinguish new csc nilpotent Lie groups up to quasi-isometry.  The key here is that  for a csc nilpotent Lie group $G$ with Lie algebra $\mathfrak{g}$, the van Est theorem \cite{van-est} provides an isomorphism of graded-commutative $\RR$-algebras $\Cohom^*(G,\RR) \simeq \Cohom^*(\mathfrak{g}, \RR)$, where the latter denotes the Lie algebra cohomology of $\mathfrak{g}$ (see e.g.~\cite[Chapter II]{guichardet-book}).
 Computations of group cohomology may therefore be pushed to the Lie algebra side, and thus reduced to problems in finite dimensional linear algebra which can be solved by a computer.
 
 \subsection{Separation by Betti numbers}
The aim of this section is to show how Corollary \ref{mainthm:cor} can be used to give new examples of non-quasi-isometric,  7-dimensional, csc, nilpotent Lie groups. The main virtue of these examples is that the use of computer algebra systems needed  is relatively light compared to the more advanced separation/non-separation results in the following section, which rely heavily on solving large systems of polynomial equations. \\
{Recall that  Gong's list \cite{gong} contains six 1-parameter (where the parameter ranges over  $\CC$ with the possible exception of finitely many values) families of complex $7$-dimensional nilpotent Lie algebras,
and nine 1-parameter families (with parameter ranging over $\RR$ with the possible exception of finitely many values) of real 7-dimensional nilpotent Lie algebras. To give the reader an idea of what these families look like we will present one example of a real 1-parameter family. We remark that we will use the notation used by Gong throughout this section, so that $\lbrace x_1,\dots, x_7\rbrace$ will always denote the basis elements of a given Lie algebra, and $\langle \cdots \rangle$ is used to denote the $\RR$-linear span of vectors. As is customary, all non-specified brackets between basis elements are implicitly set to zero.
\subsection*{$1357M$}
The family $1357M$ has the following bracket relations for $\lambda\in \RR\setminus\{0\}$:
\begin{align*}
[x_1,x_2]&=x_3, & [x_1,x_i]&=x_{i+2},\  i=3,4,5, & [x_2,x_4]=x_5,
\\
[x_2,x_6]&=\lambda x_7, & [x_3,x_4]&=(1-\lambda) x_7. &
\end{align*}
The family has lower central series
\begin{align*}
\langle x_1,x_2,x_3,x_4,x_5,x_6,x_7 \rangle \geq \langle x_3,x_5,x_6,x_7 \rangle \geq \langle x_5,x_7 \rangle \geq \langle x_7 \rangle,
\end{align*}
and Carnot algebra given by the relations
\begin{align*}
[\bar{x}_1,\bar{x}_2]&=\bar{x}_3, &  [\bar{x}_1,\bar{x}_i]&=\bar{x}_{i+1},i=3,4,5.
\end{align*}
The Carnot algebra corresponds to the 7-dimensional real Lie algebra denoted by $2457A$ in \cite{gong}.\\

 We may give a summary of the algebras, their lower central series dimensions and Betti numbers by means of  the following table:
\\
\begin{table}[!htbp]\label{Table:Summary}\resizebox{\columnwidth}{!}{\begin{tabular}{|c|c|c|c|c|c|c|}
\hline
$\CC$ & $123457I$ & $12457N$ & $1357N$ & $1357S$ & $1357M$ & $147E$
\\
\hline
$\RR$ & $123457I$ & \begin{tabular}{l|l} $12457N$ & $12457N_2$ \end{tabular} & $1357N$ & \begin{tabular}{l|l} $1357S$ & $1357QRS_1$ \end{tabular} & $1357M$ & \begin{tabular}{l|l} $147E$ & $147E_1$ \end{tabular}
\\
\hline
$\beta^n$ & 12344321 & 12344321 & 13577531&13677631 &13688631 &13799731
\\
\hline
LC & 754321 &  75421 & 7421 & 7431 & 7421 & 741
\\
\hline
$\mathfrak{Car}(\mathfrak{g})$ & $123457A$ &  \begin{tabular}{c|c} $12457L$ & $12457L_1$ \end{tabular} & $2457A$ & See below & $2457A$ & \begin{tabular}{l|l} $147E$ & $147E_1$ \end{tabular}
\\
\hline
\end{tabular}}
\vspace{0.5cm}
\caption{Summary of 1-parameter families of 7-dimensional nilpotent Lie algebras.}\end{table}
\\
The first row consists of the the complex families, the second of the real families (the complexification of each being the algebra right above it), 
and the third consists of the sequence of Betti numbers; i.e.~the dimensions of the  real cohomology groups. The fourth row consists of the dimensions of the algebras appearing in the lower central series, and the fifth row contains the Carnot algebra of the family in question}.
Here the families and their complexifications may be found listed in \cite{gong}, and the Betti numbers and lower central series dimensions  can easily be computed in e.g. Maple.  We note that \cite{Magnin} also provides a complete list of Betti numbers from which these may also be taken. We stress that we are here only using the generic Betti numbers: as one can see from e.g.~Magnin's paper, several families have a finite number of parameter-values for which the Betti numbers are different from the ones listed, but
%\footnote{the exclusion of finitely many values stems from the Gaussian elimination performed by Maple in which certain polynomials in $\lambda$ show up as denominators.}. 
since these finitely many non-generic cases are irrelevant for our purposes, we will disregard them here.
From the lower central series, the Carnot algebras are easily computed and may be compared (e.g.~using Maple) with those found in Gong's list. For most families one obtains a single (family of) Carnot algebra(s), 
but for the families $1357S$ and $1357QRS_1$ the situation is not quite as simple:   one here obtains that the family $1257S$ (with parameter $\lambda\in \RR\setminus\{1\}$) has Carnot algebra $N_{6,2,5}\oplus \RR$ when  $\lambda <1$ and $N_{6,2,5a} \oplus \RR$ when $\lambda>1$. Similarly, $1357QRS_1$ (with parameter $\lambda\in \RR\setminus\{0\}$) has associated Carnot algebra $N_{6,2,5}\oplus \RR$ when $\lambda < 0$ and $N_{6,2,5a} \oplus \RR$ when $\lambda>0$.\\
We now turn towards the quasi-isometry classification problem for members of different families.
By Pansu's theorem, the Carnot algebra is a quasi-isometry invariant,  so in most cases  Table \ref{Table:Summary} rules out the possibility that the csc nilpotent Lie groups associated with  two members of different families can be quasi-isometric. However,   the families $1357N$ and $1357M$ are seen to have the same associated Carnot Lie algebra and hence cannot be separated by Pansu's theorem, but they do have different Betti numbers, and thus separation here follows from Corollary \ref{mainthm:cor}.
%which in turn have different Betti numbers, and hence no pair of one generic algebra from each of these families can  have isomorphic cohomologies and thus the separation here follows from Corollary \ref{mainthm:cor}.}
%Since the dimensions of the lower central series dictate the dimensions of the quotient subalgebras of the Carnot algebras, one can see from the table that any two families which do not have the same lower central series dimensions cannot have isomorphic Carnot algebras and consequently the quasi-isometry classes of the associated csc Lie groups may be separated by Pansu's theorem, without explicitly computing the Carnot algebras, which thus only necessary to do for the 3 pairs which share complexifications. Thus separation may be concluded using Pansu's theorem between all families except between the families  $1357S$ and $1357QRS_1$, from which no conclusions may be made using Table \ref{Table:Summary}, and between the families $1357N$ and $1357M$ which in turn have different Betti numbers, and hence no pair of one generic algebra from each of these families can  have isomorphic cohomologies and thus the separation here follows from Corollary \ref{mainthm:cor}.
Since at most a countable number of members of each family have an associated csc Lie group which admits a lattice \cite{malcev}, this illustrates an instance where Corollary \ref{mainthm:cor} can tell apart (up to quasi-isometry) Lie groups that are not distinguished by neither the results of Pansu, nor by those of Sauer and Shalom \cite{sauer, shalom-harmonic-analysis}. 
We note that  $1357S$ and $1357QRS_1$   are the only ones for which no generic separation between the two families can be made from neither Pansu's theorem nor their Betti numbers, although Pansu's theorem of course proves partial information, in  that members whose parameter values give rise non-isomorphic Carnot algebras cannot be quasi-isometric.

%\textcolor{red}{Consequently, no conclusions about separation between these two families may be made from either Pansu's theorem nor their Betti numbers.}\\

\subsection{Separation within continuous $1$-parameter families}
Another natural question is  whether it is possible to use Corollary \ref{mainthm:cor} to find the first uncountable family of pairwise non-quasi-isometric csc nilpotent Lie groups that are not already distinguished by Pansu's teorem \cite[Th{\'e}or{\`e}me 3]{pansu}.
%{The classification in \cite{gong} identifies nine 1-parameter families of 7-dimensional nilpotent Lie algebras over $\RR$:
As can be seen from Table \ref{Table:Summary}, the nine 1-parameter families of real 7-dimensional nilpotent Lie algebras fall in three groups:

\begin{itemize}
\item[(i)] $147E$ and $147E_1$ which are Carnot;
\item[(ii)]  $123457I, 12457N, 12457N_2, 1357M$ and $1357N$ which are not Carnot, but within each family the associated family of Carnot algebras, generically, consists of isomorphic  members;
\item[(iii)] $1357S$ and $1357QRS_1$, where each family has two isomorphism classes of associated Carnot algebras.
\end{itemize}

We will now show that each family in the second class (ii) have generically isomorphic cohomology algebras, and  that each  family in the third class (iii) gives rise to cohomology algebras which fall within uncountably many isomorphism classes. For results about the class (i), see Remark \ref{rem:dim5-6}. \\

Using Maple, one can compute a linear basis for the cohomology of any given finite dimensional Lie algebra, as well as multiplication tables for for the cup product. 
%From this one sees that, within each of the 9 families listed above, the $i$-th Betti number $\beta_i(\mathfrak{g}(t)):=\dim_{\RR} \Cohom^i(\mathfrak{g}(t), \RR)$, where where $(\mathfrak{g}(\lambda))_\lambda$ denotes the family in question, is (except for a finite number of members) independent of the parameter $t$. 
As an illustration of the type of multiplication tables one encounters when doing this for the 1-paramter families, we include here two examples:

\begin{table}[!htbp]
\label{Table:1357M}
\resizebox{\columnwidth}{!}{$
\begin {array}{|cccccccc|} 
\hline
0&0&-{\frac {{t}^{4}-2\,{t}^{3}+
4\,{t}^{2}-2\,t+2}{{t}^{2}-t+1}}&0&0&0&0&-{\frac {{t}^{2}+2
}{{t}^{2}-t+1}}\\ 0&-{\frac {6\,{t}^{4}-6\,{t}^{
3}+13\,{t}^{2}-4\,t+6}{{t}^{2} \left( {t}^{2}-4\,t+4 \right) }}&0&0
&0&0&0&0\\ -{\frac {{t}^{2}-2\,t+2}{{t}^{2
}-2\,t+1}}&0&0&0&0&0&0&0\\ 0&0&-
{\frac {{t}^{3}-3\,{t}^{2}+4\,t-2}{{t}^{2}-t+1}}&0&0&0&0&-{
\frac {{t}^{2}+2}{{t}^{2}-t+1}}\\0&0&0&0&0
&0&2\,{\frac {{t}^{2}+t+1}{{t}^{2}+2\,t+1}}&0
\\ 0&0&0&0&0&-1&0&0
\\ 0&0&0&0&{\frac {{t}^{4}-2\,{t}^{3}+5\,{t
}^{2}-4\,t+6}{{t}^{4}-2\,{t}^{3}+3\,{t}^{2}-2\,t+1}}&0&0&0
\\ 0&0&0&1&0&0&0&0\\
\hline
\end {array}
$}
\vspace{0.5cm}
 \caption{Multiplication table for 3rd and 4th cohomology for the family $1357M$}
\end{table}

\begin{table}[!htbp]
\label{Table:12457N}
\resizebox{\columnwidth}{!}{$
\arraycolsep=1.4pt\def\arraystretch{2.2}
\begin {array}{|cccc|} 
\hline
-{\frac {3\,{t}^{4}-3\,{t}^{3}+32\,
{t}^{2}-10\,t+108}{3t \left( {t}^{2}-2\,t+6 \right) }}&-{\frac {
3\,{t}^{6}-6\,{t}^{5}+50\,{t}^{4}-52\,{t}^{3}+276\,{t}^{2}-120\,t+504
}{6t \left( {t}^{2}-2\,t+6 \right) }}&-{\frac {15\,{t}^{4}-30\,{
t}^{3}+158\,{t}^{2}-160\,t+432}{3t \left( {t}^{2}-2\,t+6 \right) }}&
{\frac {3\,{t}^{5}+6\,{t}^{4}+14\,{t}^{3}+100\,{t}^{2}-28\,t+360}{6
t \left( {t}^{2}-2\,t+6 \right) }}\\ {\frac {
3\,{t}^{3}+2\,{t}^{2}+7\,t+8}{2({t}^{2}-2\,t+6)}}&{\frac {3\,{t}^{
6}+5\,{t}^{5}+27\,{t}^{4}+37\,{t}^{3}+78\,{t}^{2}+42\,t+48}{4({t}^{3}-{t
}^{2}+4\,t+6)}}&-{\frac {3\,{t}^{4}+5\,{t}^{3}+5\,{t}^{2}+23\,t-16}{{
t}^{3}-{t}^{2}+4\,t+6}}&-{\frac {3\,{t}^{5}-7\,{t}^{4}-{t}^{3}-
45\,{t}^{2}-34\,t-196}{4({t}^{3}-{t}^{2}+4\,t+6)}}\\ {\frac {11\,{t}^{2}-2\,t+41}{3({t}^{2}-2\,t+6)}}&{\frac {11\,
{t}^{4}-2\,{t}^{3}+107\,{t}^{2}-12\,t+246}{6({t}^{2}-2\,t+6)}}&-{
\frac {2(11\,{t}^{2}-2\,t+41)}{3({t}^{2}-2\,t+6)}}&-{\frac {11\,{t}^{
3}-24\,{t}^{2}+45\,t-82}{6({t}^{2}-2\,t+6)}}\\ -
{\frac {3\,{t}^{3}-14\,{t}^{2}+28\,t-27}{3t \left( {t}^{2}-2\,t+6
 \right) }}&{\frac {8\,{t}^{4}-4\,{t}^{3}+57\,{t}^{2}-12\,t+126
}{6t \left( {t}^{2}-2\,t+6 \right) }}&{\frac {11\,{t}^{2}-34\,t+
108}{3t \left( {t}^{2}-2\,t+6 \right) }}&-{\frac {8\,{t}^{3}-2\,
{t}^{2}-t+90}{6t \left( {t}^{2}-2\,t+6 \right) }}\\
\hline
\end {array} 
$}
\vspace{0.5cm}
\caption{Multiplication table for 3rd and 4th cohomology for the family $12457N$}
\end{table}
The tables should be read as follows: the parameter $t$ is the parameter of the family $(\mathfrak{g}_t)_t$ in question. In Table 1, we have that $ \Cohom^3(\mathfrak{g}(t),\RR)$ is 8 dimensional with a fixed basis $e_3^1(t), \dots, e_3^8(t)$, that $ \Cohom^4(\mathfrak{g}(t),\RR)$ is also 8 dimensional with a fixed basis $e_4^1(t), \dots, e_4^8(t)$ and that $ \Cohom^7(\mathfrak{g}(t),\RR)$ is one dimensional with a fixed basis $e_7(t)$. The $(i,j)$th entry of the table is the coefficient of the product $e_3^i(t)\smile e_4^j(t)$; e.g.~we have 
$$e_3^1(t)\smile e_4^3(t)=-{\frac {{t}^{4}-2\,{t}^{3}+
4\,{t}^{2}-2\,t+2}{{t}^{2}-t+1}}e_7(t)$$ from the first table. A similar reasoning applies to Table 2.\\

We now explain how to construct isomorphisms of graded $\RR$-algebras
\[
\varphi_{t,s} \colon \Cohom^*(\mathfrak{g}(s),\RR)\to \Cohom^*(\mathfrak{g}(t),\RR).
\]
 Denote again the fixed basis of $\Cohom^*(\mathfrak{g}(\lambda),\RR)$ by $e_i^j(\lambda)$, so that for fixed $i$  the elements $e_i^j(\lambda)$ is a basis for $\Cohom^i(\mathfrak{g}(\lambda),\RR)$.   Using the fact that $\varphi_{t,s}$ should be linear and graded we know that $\varphi_{t,s}$ is given by the coefficients associated with the image of the basis elements: 
%$$ \varphi(e_j^i(s))= \smashoperator{\sum_{k=1}^{ \dim \Cohom^j(\mathfrak{g}(t),\RR)}} \alpha_{j,i}^k e_j^k(t).$$
$$ \varphi_{t,s}(e_i^j(s))= \smashoperator{\sum_{k=1}^{ \beta_i(\mathfrak{g}(t))}} \alpha_{i,j}^k e_i^k(t).$$
 In other words, we may consider $\varphi$ as a block-diagonal matrix, where the blocks are $\beta_i(\mathfrak{g}(t)) \times \beta_i(\mathfrak{g}(t))$-matrices.
 The map $\varphi$ being an $\RR$-algebra homomorphism is then encoded by  solutions to a polynomial system of equations in the variables $\alpha_{i,j}^k$ arising from the multiplication table, and one can easily check whether $\varphi$ is bijective or not by taking the determinant of the matrix corresponding to the solutions. \\
The method used to obtain solutions to the system is {as follows}, and the implementation has been done in Maple 19: All equations appearing are polynomial in the variables $\alpha_{i,j}^k$ with certain rational functions in the parameters $s$ or $t$ as coefficients. 
First the system is simplified by inspecting the equations, and replacing every equation of the form $p(s)\alpha_{i,j}^k=0$ (with $p$ a rational function in $s$) with $\alpha_{i,j}^k=0$ , and similarly for equations of the form $p(t)\alpha_{i,j}^k\alpha_{\ell,m}^n=0$. Then using the zero-product rule on equations of the second type above, along with a non-singularity assumption on all the block-matrices along the diagonal, one may in certain cases conclude which factor should be $0$. To illustrate this, consider the family $1357M$ which has $\beta_1(\mathfrak{g}(t))=3$ for all but finitely many values of $t$, and for which the corresponding system of equations contains the following three:
$$\alpha_{1,1}^1\alpha_{4,4}^6=0, \quad \alpha_{1,2}^1\alpha_{4,4}^6=0, \quad \alpha_{1,3}^1\alpha_{4,4}^6=0
$$
Since $\alpha_{1,1}^1=0, \alpha_{1,2}^1=0$ and $\alpha_{1,3}^1=0$ would cause the first block-matrix to be singular, we may conclude that $\alpha_{4,4}^6=0$.
These zero-values are then substituted into the system which is then reduced. The procedure is then continued until no further conclusions can be made this way. 
{After this, we tried solving the system. In case the system was still too complex to solve in reasonable time, we would make qualified guesses on values of variables based on solutions of parts of the system. This process could be repeated several times. After a full solution was obtained, we would set any remaining free variables to 1, compute the determinants of the diagonal matrices, and verify the solution by substituting it into the system. For all families of the second class (ii), this provided a solution which was an isomorphism for all but finitely many values of the parameter, so that we have the following:
\begin{prop}
\label{Prop:CohmIso}
Within each of the 7-dimensional one-parameter families $123457I$, $12457N$, $12457N_2$, $1357M$ and $1357N$ of nilpotent Lie algebras, all members, except for perhaps finitely many, have isomorphic cohomology algebras. 
\end{prop}
Thus, Corollary \ref{mainthm:cor} cannot be used to tell  the Lie groups associated with the members of these families apart up to quasi-isometry.}
We remark that obtaining solutions to the equation systems described above carries a high degree of complexity. For example, the system of equations for the family 1357M consists of 1193 unique (but possibly equivalent) equations, with up to 11 summands of the form $p(t) \alpha_{j,i}^k\alpha_{\ell,m}^n$, where the coefficients $p(t)$, as already mentioned, are rational functions in $t$. In contrast to this, the system obtained from the family $12457N$ consists of 94 equations, each of which contains up to 17 summands.
The concrete formulas for the graded $\RR$-algebra isomorphisms $\varphi_{t,s} \colon \Cohom^*(\mathfrak{g}(s),\RR)\to \Cohom^*(\mathfrak{g}(t),\RR)$  are not included here due to their sizes; for instance for the family $12457N$ the $\varphi_{t,s}$ is given by $20\times 20\,$-matrix, and the expression below is an example of one of its prototypical non-zero entries: 
\begin{align*}
& \quad {\tfrac {3(t-s)}{ \left( t+1 \right)  \left( 6\,{t}^{3}-
19\,{t}^{2}-16\,t-167 \right)  \left( s+1 \right)  \left( {s}^{3}-2\,{
s}^{2}+17\,s-2 \right) }  } \Big[ \left( -4+{t}^{4}-\tfrac{7}{6}t^3-{
\tfrac {20}{3}}{t}^{2}+{\tfrac {37}{6}t} \right) {s}^{4} +
\\
&\left.+ \left( 2-\tfrac{7}{6}{t}^{4}-{\tfrac {62}{3}{t}^{3}}-{\tfrac {143}{6}{t}^{2}}-17\,t
 \right) {s}^{3} \right.  + \left( {\tfrac {207}{2}t}-{\tfrac{175}{3}}-{\tfrac {
20}{3}{t}^{4}}-{\tfrac {143}{6}{t}^{3}}-{\tfrac {226}{3}{t}^{2}}
 \right) {s}^{2}+
\\
& + \left( -{\tfrac{400}{3}}+{\tfrac {37}{6}{t}^{4}}-17
{t}^{3}+{\tfrac {207}{2}{t}^{2}}+{\tfrac {640}{3}t} \right) s
 -{
\tfrac {400}{3}t}-4\,{t}^{4}+2\,{t}^{3}-{\tfrac{119}{3}}-{\tfrac {175}{3}{t}^{2}}
\Big]
\end{align*}
Maple worksheets performing this procedure with full solutions  (and the correct guesses written down) can be obtained from the authors upon request; see also Remark \ref{rem:urlrem}.
\\

For the two  families in (iii), we will now show that the picture is quite different, in that the associated families of cohomology algebras fall into uncountably many isomorphism classes.
{As noted earlier, the families are not separable by Pansu's theorem, as it can see only two quasi-isometry classes.
However, by appealing to Corollary \ref{mainthm:cor} we do in fact obtain uncountably many classes:}

\begin{theorem}
\label{Thm:NonIso}
Let $(\mathfrak{g(\lambda}))_{\lambda}$ denote either of the families $1357S$ and $1357QRS_1$. Then for every parameter value $\lambda_0$ there exist at most finitely many other parameter values $\lambda$ for which  $\Cohom^*(\mathfrak{g}(\lambda),\RR)\simeq  \Cohom^*(\mathfrak{g}({\lambda_0}),\RR)$. In particular, the 1-parameter families of csc nilpotent Lie groups associated with $1357S$ and $1357QRS_1$ fall into uncountably many quasi-isometry classes.
\end{theorem}

{
To prove Theorem \ref{Thm:NonIso}, one could of course again implement the equations satisfied by a potential isomorphism in Maple with the hope that there are no solutions. However, doing so did not lead to any conclusion as an attempt to solve the system did not terminate within a reasonable time frame, and any guesses added to the system would  make it unable to be used to disprove the existence of isomorphisms}.  Instead, we will require a bit of deformation theory to prove Theorem \ref{Thm:NonIso}. For our purpose a very ad hoc introduction suffices, but the interested reader may consult \cite{Gerstenhaber-deformation} for a more thorough presentation. \\

Given an algebra $A$, a one-parameter (formal) deformation of $A$, is a formal power series $F=\sum_{i=0}^\infty f_it^i$ such that the coefficients are linear maps $f_i\colon A\otimes A \to A$ and $f_0$ is the multiplication in $A$.  The first nontrivial $f_i$ with $i\geq 1$ is called the \emph{infinitesimal} of the deformation.  The power series $F$ defines a product on the space of formal power series $A[[t]]$ with coefficients in $A$, by defining $a\ast b:= \sum_{i=0}^\infty f_i(a,b)t^i$ for $a,b\in A$ and extending to all of $A[[t]]$.  If this product is associative, the infinitesimal $f_{i_0}\colon A\otimes A \to A$ becomes a Hocschild $2$-cocycle. There is of course a notion of triviality and equivalence between deformations, and these concepts turn out to be intimately linked with triviality of the class of $f_{i_0}$ in $\Cohom^2(A,A)$  (see e.g.~\cite{Gerstenhaber-deformation}).  To this end, we recall that $f_{i_0}$ is a Hochschild coboundary if there exists a linear map $g\colon A\to A$ such that $f_{i_0}=\del_1 g$, where
\[
\del_1(g)(a,b) :=ag(b)-g(ab) +g(a)b.
\]
If the algebra $A$ comes equipped with a topology and $\lambda\in \RR$ is such that the series $F_{t}(a,b)=\sum_{i=0}^\infty f_i(a,b)t^i$ converges for $t=\lambda$, then  $(A, F_\lambda)$ becomes an algebra (associative if $F$ is so) called the specialisation of $F$ at $\lambda$. In general, a (formal) deformation may not have any specialisations, but we will show below that when $(\mathfrak{g}_\lambda)_\lambda$ is one of the families $1357S$ or $1357QRS_1$, then the family of cohomology algebras $(\Cohom^*(\mathfrak{g}_\lambda,\RR))_\lambda$ may be viewed as specialisations of a formal deformation. We are of course interested in whether  or not the cohomology algebras $(\Cohom^*(\mathfrak{g}_\lambda,\RR))_\lambda$ are pairwise isomorphic, which is an a priori different notion of triviality than the one considered in formal deformation theory,  but as the following two lemmas will show, triviality/non-triviality of the $2$-cocycle $f_1$, which turns out to be  the infinitesimal, again plays an important role.

\begin{lemma}
\label{Lemma:Diff}
Let $A$ be a finite-dimensional real algebra, equipped with the Euclidean topology, and let $F$ be a deformation of $A$. Assume that there is a neighbourhood of zero, $U\subset \RR$, on which $F$ admits specialisations $\lbrace A_\lambda \rbrace_{\lambda\in U}$,  such that there exists a family of isomorphisms $\lbrace \varphi_\lambda\colon A_\lambda \to A_0 \rbrace_{\lambda\in U }$ with  $\varphi_0$ being the identity. Assume further that   $\frac{d}{d\lambda}\varphi_\lambda(a)\vert_{\lambda=0}:= \lim_{\lambda\to 0} \frac{\varphi_{\lambda}(a)-a}{\lambda}$ exists for all $a\in A$. Then the map $g\colon A \to A$ defined by $g(a)=\frac{d}{d\lambda}\varphi_\lambda(a)\vert_{\lambda=0}$ satisfies $\partial_1g=f_1$;  thus $f_1$ is a Hochschild coboundary. 

\begin{proof}
This is an easy consequence of taking the derivative at zero on both sides of the equation
$\varphi_\lambda(a)\varphi_\lambda(b)=\varphi_\lambda(ab)+ \sum_{k=1}^\infty \varphi_\lambda(f_i(a,b))\lambda^i $, noting that $\varphi_0$ is the identity so that the Leibniz rule gives 
\[
\frac{d}{d\lambda}(\varphi_\lambda(a)\varphi_\lambda(b)) \vert_{\lambda=0}=ag(b)+g(a)b,
\]
and  that
$$\tfrac{d}{d\lambda}\varphi_\lambda(f_i(ab)\lambda^i)\vert_{\lambda = 0}=\left\lbrace \begin{matrix}
g(ab) & i=0
\\
f_1(a,b) & i=1
\\
0 & i>1
\end{matrix}\right.
$$
\end{proof}
\end{lemma}

\begin{lemma}
\label{Lemma:Solutions}
Let $V$ be a  finite dimensional vector space with a fixed basis $e_1,\dots, e_n$, and let $U\subset \RR$ be an open interval such that for each $\lambda\in U$ one has a product $\ast_\lambda$ on $V$ turning it into an associative real algebra $A_\lambda$. Assume moreover that the coefficients of $e_k\ast_\lambda e_l$ relative to  the fixed basis are given by  rational functions  in $\lambda$.  Fix $\lambda_0\in U$ and assume that for infinitely many $\lambda\in U$ there exists  an algebra-isomorphism from $A_\lambda$ to $A_{\lambda_0}$. Then there exists an open set $V\subset U$ and $n\times n$-matrices $((\alpha_{ij}(\lambda))_{ij})_{\lambda\in V}$ such the associated maps in $\operatorname{End}(V)$ are algebra-isomorphisms $A_{\lambda} \simeq A_{\lambda_0}$ and such that  $\lambda\mapsto \alpha_{ij}(\lambda)$ is differentiable for every $1\leq i,j \leq n.$

\end{lemma}

The lemma is a consequence of the so-called Nash cell decomposition of semi-algebraic sets; we refer to Chapter 2 of \cite{Alg-geom-book} for the basics of the this theory.

\begin{proof}
Denote by $S\subset \RR^{n^2+1}$ the set of $(\alpha_{ij}(\lambda),\lambda)$ where $\lambda \in U$ and $(\alpha_{ij}(\lambda))_{i,j}\in \MM_n(\RR)$ implements and algebra-isomorphism $A_\lambda \simeq A_{\lambda_0}.$ Since  the scalar coefficients of products are rational functions in $\lambda$, the  system of equations  which encodes that $(\alpha_{ij}(\lambda))_{ij}$ is an algebra-isomorphism  consists of polynomials in $\alpha_{ij}(\lambda)$ and $\lambda$, inequations of the form $\lambda\neq x$ where $x$ is a root in a polynomial appearing in a denominator of one of the rational functions, and the polynomial inequation $\det\left(\left(\alpha_{ij}(\lambda)\right)_{ij}\right)\neq 0$. Since $U$ is an interval, the set $S$ is therefore a semi-algebraic set in $\RR^{n^2+1}$ (see  \cite{Alg-geom-book} for the relevant definitions), and by \cite[Proposition 2.9.10]{Alg-geom-book}, $S$ admits a decomposition into a disjoint union of finitely many submanifolds 
\[
S=A_1\sqcup \dots \sqcup A_n
\]
 such that each $A_i$ is either a singleton set or diffeomorphic to an open hypercube $(0,1)^{d_i}\subset \RR^{d_i}$. Since we assume that solutions exist for infinitely many $\lambda\in U$, we may, without loss of generality, assume that $A_1$ contains at least two points, $p_1:=((\alpha_{ij}(\lambda_1))_{ij},\lambda_1)$ and $p_2:=((\alpha_{ij}(\lambda_2))_{ij},\lambda_2)$, with $\lambda_1< \lambda_2$.
 Let now $\varphi\colon A_1\to (0,1)^{d_1}$ be a diffeomorphism and let $\gamma\colon [0,1] \to (0,1)^{d_1}$ be the straight line segment between $\varphi(p_1)$ and $\varphi(p_2)$, and let $\pi_{n^2+1}\colon \RR^{n^2+1}\to \RR$ denote the projection onto the last coordinate. Then the map $f\colon [0,1] \to \RR$ given by $f(t)=\pi_{n^2+1}\circ \varphi^{-1}\circ \gamma$ is continuous and differentiable on $(0,1)$ and satisfies $f(0)=\lambda_1$ and $f(1)=\lambda_2$. Hence there is a $\hat{t}\in (0,1)$ such that $f'(\hat{t})\neq 0$, and thus $f$ is invertible on a neighbourhood of $\hat{t}$. On an open interval $I$ around $\hat{\lambda}:=f(\hat{t})$,  we may  therefore consider $\psi=\varphi^{-1}\circ \gamma \circ f^{-1} \colon I\to \RR^{n^2+1}  $ and we have that $(\alpha_{ij}(\lambda),\lambda)\coloneqq (\psi(\lambda))$ is a solution, and each $\alpha_{ij}(\lambda)$ is differentiable.
\end{proof}

We now return to the concrete algebras we wish to handle, namely the cohomology algebras arising from the 1-parameter families  $1357S$ and $1357QRS_1$.  In each case, using Maple one obtains a basis for the total cohomology $\Cohom^*(\mathfrak{g_\lambda}, \RR)$ and multiplication tables for the cup product, in which all the coefficients are rational functions in $\lambda$.  For all but finitely many parameter values the vector space $\Cohom^*(\mathfrak{g}_\lambda, \RR)$ has a fixed dimension, and
for these parameter values  we may therefore consider the algebras $\Cohom^*(\mathfrak{g_\lambda}, \RR)$ as a $1$-parameter family of products $\smile_\lambda$ on a fixed vector space $V$ with a fixed basis $e_1,\dots, e_n$, thus putting us within the scope of  Lemma \ref{Lemma:Solutions}. Denote by $p_{ij}^k$ the rational functions describing the product of $e_i$ and $e_j$; i.e.~
\[
e_i \underset{\lambda}{\smile} e_j =\sum_{k=1}^n p_{ij}^k(\lambda)e_k.
\]
 Since the coefficients are rational functions, they admit  a power series expansion for all but finitely many $\lambda$, so we may construct a deformation around a chosen, nonsingular, $\lambda_0$ by shifting the variable and using the Maclaurin expansion; i.e.~by setting
 \[
 f_m(e_i,e_j) \coloneqq \sum_{k=1}^n  \left.\dfrac{d^m}{dt^m}p_{ij}^k \right\vert_{t=\lambda_0} \frac{1}{m!}  e_k
 \]
 we obtain a formal power series $F(a,b)=\sum_{m=0}^\infty f_m(a,b) t^m$ which converges for $t$ in a neighbourhood of zero  and such that the specialisation at $t=0$ is exactly $\Cohom^*(\mathfrak{g}_{\lambda_0}, \RR)$. Thus, $F$ constitutes a (formal) deformation of the algebra $\Cohom^*(\mathfrak{g}_{\lambda_0}, \RR)$, and implementing the above in Maple, it turns out that for the two families in question, $f_1$ is  non-zero for all but finitely many values of $\lambda_0$, and thus the infinitesimal here.

 In the light of Lemma \ref{Lemma:Diff}, we now need to investigate if $f_1$ is a coboundary in each of the two cases, i.e.~if there exists a linear map $g\colon \Cohom^*(\mathfrak{g}(\lambda_0),\RR)\to \Cohom^*(\mathfrak{g}(\lambda_0),\RR)$ such that
 \begin{align}\label{eq:coboundary}
 f_1(x,y)=\partial_1(g)(x,y)\coloneqq g(x)\underset{ \hspace{0.1cm}\lambda_0}{\smile} y - g(x\underset{  \hspace{0.1cm} \lambda_0}{\smile} y)+ x\underset{  \hspace{0.1cm}\lambda_0}{\smile} g(y). 
 \end{align}
 If such a map exists, it is uniquely determined by the scalars $b_{ij}$ satisfying
 \[
 g(e_i)=\sum_{j=1}^n b_{ij}e_j,
 \]
 and setting $x=e_i$ and $y=e_j$ in \eqref{eq:coboundary} yields

 \begin{align*}
 \sum_{k=1}^n  \left.\dfrac{d}{dt}p_{ij}^k \right\vert_{t=\lambda_0}   e_k &=\sum_{k=1}^n b_{ik}e_k \underset{ \hspace{0.1cm}\lambda_0}{\smile} e_j - g\left(\sum_{k=1}^n p_{ij}^k(\lambda_0)e_k\right) + \sum_{k=1}^n b_{jk} e_i \underset{ \hspace{0.1cm}\lambda_0}{\smile}e_k\\
&= \sum_{k,l=1}^n b_{ik}p_{kj}^l(\lambda_0)e_l  - \sum_{k,l=1}^n p_{ij}^k(\lambda_0) b_{kl}e_l +\sum_{k,l=1}^n b_{jk}p_{ik}^l(\lambda_0)e_l.
  \end{align*}
Comparing coefficients of the basis vectors now gives a system of linear equations in the variables $b_{ij}$ and, by construction, $f_1$ is a Hochschild 1-coboundary if and only if this system has a solution. {Both of the families $1357S$ and $1357QRS_1$ have 34-dimensional cohomology algebras, and the resulting linear systems are much bigger than what can reasonably be solved by hand (the systems contain, respectively, 7066 and 6388 equations), but using Maple, it has been verified}, that for all but a most finitely many values of $\lambda_0$ the corresponding system of equations does not have any solutions.

\begin{proof}[Proof of Theorem \ref{Thm:NonIso}]
Let $(\mathfrak{g}_\lambda)_\lambda$ be either of the two families.  For all but finitely many parameter values, $\Cohom^*(\mathfrak{g}_\lambda,\RR)$ has a fixed dimension, so let $\lambda_0$ be such a generic parameter value and choose an open interval $U$ around $\lambda_0$ consisting of generic parameter values. As argued in the beginning of this section, we may then consider the algebras $(\Cohom^*(\mathfrak{g}_\lambda,\RR))_{\lambda\in U}$ as a family of algebra structures on a fixed vector space with a fixed basis, with the property that the coefficients of products of basis vectors are given by rational functions in $\lambda$. So, if  $\Cohom^*(\mathfrak{g}_{\lambda_0},\RR)\simeq \Cohom^*(\mathfrak{g}_\lambda,\RR)$ for infinitely many other parameter values, then Lemma \ref{Lemma:Solutions} provides a $\hat{\lambda}$, an open neighbourhood $V$ around $\hat{\lambda}$ and an isomorphism $\varphi_{\lambda}=(\alpha_{ij}(\lambda))_{ij}\colon \Cohom^*(\mathfrak{g}_\lambda,\RR) \to \Cohom^*(\mathfrak{g}_{\lambda_0},\RR)$ for each $\lambda\in V$ such that every $\alpha_{ij}(\lambda)\colon U\to \RR$ is differentiable. 
{We remark that since $\dim_{\RR}\Cohom^*(\mathfrak{g}_{\lambda}, \RR)$ 
only takes finitely many values, and  since we constructed a formal deformation of $\Cohom^*(\mathfrak{g}_{\lambda_0},\RR)$ for all but finitely many $\lambda_0$, and that the associated infinitesimals were shown to be Hochschild 1-coboundaries for only finitely many $\lambda_0$, we may therefore assume that $V$ does not intersect any of these finite sets of parameter values. } 
 Upon replacing $\varphi_{\lambda}$ with $\varphi_{\lambda}\circ \varphi_{\hat{\lambda}}^{-1}$ we may assume that $\varphi_{\hat{\lambda}}$ is the identity. Considering $(\Cohom^{*}(\mathfrak{g}_\lambda,\RR))_{\lambda\in V}$ as specialisations of a formal deformation $F(-,-)=\sum_{k=0}^\infty f_k(-,-)t^k$ of $\Cohom^*(\mathfrak{g}_{\hat{\lambda}},\RR)$ as described above, 
 Lemma \ref{Lemma:Diff} now implies that the first term $f_1$  is a Hochschild 1-coboundary, contradicting what what was established in the paragraph preceding the proof. 
\end{proof}

\begin{remark}
We remark that none of the results related to Theorem \ref{Thm:NonIso} involve graded algebras or maps, but that we in fact show that no uncountable subfamily whose cohomology algebras are isomorphic, as algebras,  exists. Since graded isomorphisms are also encoded by a system of polynomial equations, the proof of Lemma \ref{Lemma:Solutions} carries over verbatim to the obvious graded restatement, but this has not been relevant for our purposes.
\end{remark}

\begin{rem}\label{rem:dim5-6} {As pointed out by an anonymous referee, it would be of interest to find two non-isomorphic Carnot Lie algebras, which have isomorphic cohomology algebras, thus providing examples of csc nilpotent Lie  groups that are distinguished by Pansu's result, but not by their associated cohomology algebras.
One such example occurs when considering the family 147E of class (i). Since the family  is Carnot, the associated Lie groups can be distinguished by Pansu's result, and the method used to compute isomorphisms for Proposition \ref{Prop:CohmIso} yields a graded isomorphism between the cohomology of all but finitely many members of the family, thus providing such an example.
\\
Performing a similar check for the five (finite) families  of Carnot Lie algebras of dimension at most $6$, whose members share the same Betti numbers,  shows that such a result cannot be obtained in dimension lower than seven. In all cases but for those Lie algebras which have the same complexification, this can be seen directly by comparing the ranks of the multiplication maps. For the pairs which share their complexification, it turns out that the associated system of equations cannot have any non-singular solution.
\\
We have also attempted to check whether or not the cohomology algebras associated to the last remaining family, 147E$_1$, are generically isomorphic, but this has  been inconclusive. The method used to prove non-isomorphism for the algebras in class (iii), shows that the infinitesimals of the associated formal deformations are (generically) coboundaries, and hence we cannot conclude that the algebras are not isomorphic. On the other hand, all attempts to find a solution to the system of equations arising when applying the method used to prove isomorphism in cohomology for the families in class (ii) have thus far been inconclusive, since we could not reduce the system to a point where no guesses have to be made, and any combination of guesses we have tried have caused  the resulting system to have no solutions.  However, as pointed out by  an anonymous referee, it is actually possible to give a complete Hochschild cohomological  characterisation of  when all (but finitely many) members of a family $\Cohom^*(\mathfrak{g}_\lambda, \RR)$ are pairwise isomorphic or not, which is likely to provide a way of settling this remaining issue, and also provide an alternative way of proving Proposition \ref{Prop:CohmIso} requiring less computational effort, since one would have to consider a system of linear equations rather than a system of polynomial equations. 
However, this is beyond the scope of the present section, but we are grateful to the referee for providing us with this perspective on the problem.
}
\end{rem}

%\textcolor{red}{
%\begin{rem}
%An anonymous referee has most satisfactory pointed out that one may in fact obtain that for deformations of algebraic structures, such as those dealt with in the final section, one may prove that if the derivative of the differentiation map is a Hochschild 2-coboundary, then the specialisations are isomorphic except for a finite subset, and conversely, if it is not a 2-coboundary, then the isomorphism classes are all finite. Implementing this result, would provide a much nicer proof of Proposition \ref{Prop:CohmIso}, requiring a significantly smaller computational effort, and would also allow one make the same conclusion for the remaining family, $147E_1$. Using this, a complete summary of the results would be as follows:
%\\
%\resizebox{\columnwidth}{!}{\begin{tabular}{|l|l|l|l|} \hline
%Class & Separation between families & Isomorphic Carnot algebras in families & Isomorphic Cohomology \\ \hline
%(i) & By Pansu & No & Yes \\ \hline
%(ii) & By Pansu and Corollary \ref{mainthm:cor}  & Yes & Yes \\ \hline
%(iii) & None proven & Yes & No \\ \hline
%\end{tabular}}
%where all results except the conclusion on the cohomology rings of algebras of class (i) may be concluded with the results provided by the authors.
%\end{rem}
%}

\begin{rem}\label{rem:urlrem}
Maple worksheets verifying the computer assisted claims made above can be obtained from the authors upon request or, at the time of writing, be downloaded via  \url{www.imada.sdu.dk/~thgot/}.
\end{rem}

%\bibliographystyle{alpha}
%\bibliography{ny-lit}
\def\cprime{$'$} \def\cprime{$'$}

\end{document}